\documentclass[a4paper,times]{amsart}

\usepackage{amsthm,amsmath,amssymb}

\usepackage{tikz}
\usetikzlibrary{patterns}
\usetikzlibrary{calc,hobby}
\usetikzlibrary{decorations.pathmorphing,patterns,decorations,shapes,arrows, intersections,matrix,fit,calc,trees,positioning,arrows,chains, shapes.geometric,shapes,angles,quotes,fit,math}
\usetikzlibrary{arrows,positioning}

\newtheorem{theorem}{\rm\bf Theorem}[section]
\newtheorem{proposition}[theorem]{\rm\bf Proposition}

\newtheorem{corollary}[theorem]{\rm\bf Corollary}

\newtheorem*{theorem*}{Theorem}
\newtheorem*{theorem 1}{\rm\bf Proposition 1}
\newtheorem*{theorem 2}{\rm\bf Proposition 2}

\theoremstyle{definition}
\newtheorem{definition}[theorem]{\rm\bf Definition}

\theoremstyle{remark}
\newtheorem{remark}[theorem]{\rm\bf Remark}

\newtheorem{example}[theorem]{\rm\bf Example}

\def\R#1{\mathbb{R}^{#1}}
\def\Ring{R}
\def\Field{K}
\def\Modul{M}

\def\half#1#2{\begin{matrix}\frac{#1}{#2}\end{matrix}}

\def\F{\mathbb{F}}

\def\Tree#1{\langle #1\rangle}

 \DeclareMathOperator{\Span}{Span}

\DeclareMathOperator{\trace}{tr}

\begin{document}

%\frontmatter
\title{The universality of one half in commutative non\-as\-sociative algebras with identities}

\author{Vladimir G. Tkachev}
\address{Department of Mathematics, Link\"oping University\\ Link\"oping, 58183, Sweden}
\email{vladimir.tkatjev@liu.se}

\begin{abstract}
In this paper we will explain an interesting phenomenon which occurs in general nonassociative  algebras. More precisely,
we establish that any finite-dimensional   commutative nonassociative algebra over a field satisfying an identity always contains $\frac12$ in its Peirce spectrum. We also show that the corresponding $\frac12$-Peirce module satisfies the Jordan type fusion laws. The present approach is based on an explicit representation of the Peirce polynomial for an arbitrary algebra identity. To work with fusion rules, we develop the concept of the Peirce symbol and show that it can be explicitly determined for a wide class of algebras.  We also illustrate our approach by further applications  to genetic algebras and algebra of minimal cones  (the so-called Hsiang algebras).
\end{abstract}



\maketitle

%\begin{footnotesize}
%\tableofcontents
%\end{footnotesize}

\section{Introduction}

Algebras whose associativity is replaced by identities were a central topic in mathematics in the 20th century, including the classical theory of Lie and Jordan algebras. Recall that an algebra is called Jordan if any two elements $y,z\in A$ satisfy the following two identities:
\begin{align}\label{Jai}
zy-yz=0,\\
z((zz)y)-(zz)(zy)=0.\label{Jai2}
\end{align}
The Peirce decomposition relative to an algebra idempotent is an important tool in the structure study of any nonassociative algebra. For example, the multiplication operator by an idempotent in a Jordan algebra is diagonalizable and the corresponding Peirce decomposition (relative to an idempotent $c$) into invariant subspaces
\begin{equation}\label{halfpeirce}
A=A_c(0)\oplus A_c(1)\oplus A_c(\half12)
\end{equation}
is compatible with the multiplication in the sense that the multiplication of eigenvectors is described by certain  multiplication rules, also known as fusion laws. In particular,
\begin{equation}\label{halffusion}
A_c(\lambda)\,A_c(\half12)\subset {A_c(\lambda)}^\bot, \qquad \forall \lambda \in\{0,\half12,1\}.
\end{equation}

To formulate the main  result of our paper, let us briefly recall some well-known relevant  concepts; see however the concise definitions and motivating examples  in the next sections. Starting with a univariate algebra identity $P(z)=0$, its Peirce polynomial $\varrho_c(P,t)$ is obtained from the linearization of $P$ at $z=c$, where $c$ is a nonzero algebra idempotent. The key observation here is that the first linearization of $P$ is essentially a polynomial in $L_c$, where $L_c:x\to cx$ is the multiplication operator (adjoint at $c$). This implies that  for any eigenvector of $L_c$ non-collinear with $c$, its eigenvalue $\lambda\in\sigma(L_c)$ must annihilate the Peirce polynomial: $\varrho_c(P,\lambda)=0$. This yields an a priori inclusion
$$
\sigma(L_c)\subset \sigma(P,c):=\{t\in \Field:\varrho_c(P,t)=0\},
$$
where the latter zero locus is called the \textit{Peirce spectrum} of $P$ at $c$. Therefore, it is natural to think of the values in $\sigma(P,c)$ as eigenvalues of $L_c$, some having maybe multiplicity zero.

For example, any Jordan algebra $A$ satisfies $P(z):=zz^3-z^2z^2=0$ (a specialization of \eqref{Jai2} for $y=z$) and it is well known that $\varrho_c(P,\lambda)=\lambda(2\lambda-1)(\lambda-1)$, thus $\sigma(P,c)=\{0,\frac12,1\}$; see also a derivation and discussion in Example~\ref{ex:Jordan} below.

Both the explicit form and the structure properties of $\varrho_c(P,\lambda)$ may change drastically depending on an algebra identity $P$ and a choice of an idempotent $c$. It is the main goal of the present paper to establish a remarkable property that the Peirce eigenvalue $\lambda=\frac12$ and the corresponding fusion laws \eqref{halffusion} are universal in the following natural sense.

\begin{theorem}\label{th1}
Let $A$ be a finite dimensional commutative nonassociative algebra over a field of characteristic $\ne2,3$ and let $A$ satisfy a nontrivial weighted polynomial identity $P(z)=0$ in one nonassociative indeterminate $z$. Then the following holds:
\begin{itemize}
\item[(A)]
$\frac12\in\sigma(P,c)$   for any idempotent $c\ne0$;
\item[(B)]
if additionally $c$ is  semi-simple and $\lambda$ is single  root of the Peirce polynomial $\,\varrho(P,t)$ then
\begin{equation}\label{fusion12}
A_c(\lambda)A_c(\half12)\subset A_c(\lambda)^\bot:=\bigoplus_{\nu\in \sigma(c), \nu\ne \lambda}A_c(\nu).
\end{equation}
\end{itemize}
\end{theorem}

A few remarks are worth noting at this time. First, note that the above results are valid for nontrivial identities $P(z)$ depending on \textit{one} nonassociative indeterminate (see Section~\ref{sec:prelim} for a further discussion and  Section~\ref{sec:trees} for exact definitions). If an algebra satisfies an identity in several nonassociative indeterminates, one can obtain a univariate algebra identity by  substituting a fixed variable (or arbitrary polynomials in one fixed variable, in general) for all indeterminates. For example, the substitution $y=z$ in \eqref{Jai2} yields a nontrivial identity \eqref{paai}, while the same substitution in the commutativity identity \eqref{Jai} amounts to the \textit{trivial} identity $0=0$.

Next, note that the claim (A) of Theorem~\ref{th1} is  quite natural and  expected  in the following sense: it was shown in \cite{KrTk18a} that  the Peirce spectrum of a \textit{generic} commutative nonassociative algebra, i.e. an algebra with maximal possible finite number of idempotents ($=2^{\dim A}$), does not contain $\frac12$. On the other hand, the variety of all nonassociative algebras on a finite-dimensional vector space $V$ can be identified with the tensor product $\mathcal{V}=V^*\otimes V^*\otimes V$, where generic algebras form a dense in Zariski topology subset $\mathcal{V}_{gen}$ of $\mathcal{V}$. In this picture, a subset of algebras with a fixed identity can be understood as a subvariety of $\mathcal{V}$. Then (A) in Theorem~\ref{th1}  combined with  results of \cite{KrTk18a} implies that  this subvariety must be \textit{non-generic}, i.e. lie in $\mathcal{V}\setminus \mathcal{V}_{gen}$.

Beside the Jordan and general power associative algebras, the results of Theorem~\ref{th1} are  known in the following particular cases.
\begin{itemize}
\item[1)] For train baric algebras of general rank (involving either principal powers $x^1=x$, $x^{n+1}=xx^n$ or plenary powers $x^{[1]}=x$, $x^{[n+1]}=x^{[n]}x^{[n]}$),\, the presence of the Peirce number  $\frac12$ in the algebra spectrum and some analogues of \eqref{fusion12} were established by Guzzo \cite{Guzzo94} and  Guti\'errez Fern\'andez~\cite{Fernand00}. These classes of algebras have a particular interest for mathematical genetics. \

\item[2)] For metrized nonassociative algebras, i.e. algebras admitting a nondegenerate symmetric bilinear form satisfying the associating condition \eqref{Qass}, some extremal properties of the Peirce number  $\frac12$ and an analogue of \eqref{fusion12} for $\lambda=\frac12$ were recently established in \cite{Tk18a}, \cite{Tk18b} and Proposition~6.7 in \cite{KrTk18a}.
\end{itemize}

Another motivation for the above results comes from Majorana and axial algebras \cite{Ivanov09}, \cite{Ivanov15}, \cite{HRS15} with the most prominent example being the   real 196883-dimensional Conway-Griess-Norton algebra $V_\mathbb{M}$ of the Monster sporadic  simple group $\mathbb{M}$ \cite{Griess82}, \cite{Norton94}. All these algebras are commutative nonassociative and satisfy certain $\mathbb{Z}/2$-graded fusion laws that makes a part of their structure theory similar to that of the classical Jordan algebras. But in contrast to the Jordan algebra case, the Majorana algebras (in particular, the Conway-Griess-Norton algebra) are generated by idempotents with the spectrum $\{0,1,\frac14,\frac1{32}\}$ consisting of four elements. Greiss proved in \cite{Griess85} that there is no nontrivial homogeneous polynomial identity for $V_{\mathbb{M}}$  of degree $5$ except for the commutativity law \eqref{Jai}. He also asked in \cite{Griess85}, \cite{Griess2010} if there are some (relevant) homogeneous polynomial identities for $V_{\mathbb{M}}$? Combining the absence of $\frac12$ in the algebra spectrum with Theorem~\ref{th1} suggests that the answer on  Greiss' question should be negative at least for  identities understood in the  sense of the present paper.

Furthermore, note  that a general axial algebra is generated by idempotents with a priori arbitrary Peirce spectrum as, for example, axial algebras of Jordan type with Peirce spectrum $\{0,\eta,1\}$. It follows from the recent results of \cite{HSS18}, \cite{HSS18b}, \cite{DeMedts17} that the case $\eta=\frac12$ is very distinguished and leads to a different behaviour of the corresponding  algebras.

\medskip
Our paper is organised as follows. Section~\ref{sec:prelim} recalls elementary facts on nonassociative algebras and the Peirce decomposition. Section ~\ref{sec:trees} gives basic definitions of free nonassociative algebras and introduces the formalism of complete binary trees. Using this formalizm, the linearization of a nonassociative monomial $z^\alpha$ can be understood as a certain evaluation on a labeled  binary tree. Note that the connection between nonassociative algebras and complete binary trees is not new and exploited since Etherington's papers \cite{Etherington49}, \cite{Etherington60} in genetic algebras with primarily emphasis on combinatorial structure and enumerations; see also \cite{Lyubich92} and a very recent paper of Mallol and Varro \cite{Mallol17}. But as far as we know this approach have not been studied systematically for the Peirce decomposition in the general case.   In Section~\ref{sec:Peirce} we define the Peirce operator associated with a binary tree and study its basic properties. In particular, Proposition~\ref{pro:half} is an important ingredient in the proof of the claim (A) of Theorem~\ref{th1}. The second order linearizations of  arbitrary nonassociative monomials are studied in Section~\ref{sec:second}. The main result of this section is Proposition~\ref{th311} establishes a relation for the second order linearization in terms of the so-called Peirce symbol  $\frak{D}(z^\alpha;a,b,p)$ playing a fundamental role in the proof of the fusion laws \eqref{fusion12}. We also obtain explicit formulae for the Peirce symbol in some special cases. The Peirce polynomial and the Peirce spectrum for a general algebra with a weighted identity are studied in Section~\ref{sec:Peircepol}. Here we also finish the proof of Theorem ~\ref{th1}. Certain nontrivial univariate polynomial identities $P(z)=0$ have identically zero Peirce polynomial, thus saying nothing  about the algebra spectrum. Such degenerate situation may appear for identities of degree at least four; some examples can be found in \cite{Bayara10} and \cite{EldLabra}. We study the degenerated  identities  in Section~\ref{sec:trivial}.

In the remainder of this paper, we obtain some further applications of our method. In particular, we revisit in Section~\ref{sec:appl} principal and plenary train algebras of general rank and give a short derivation of some recent results due to  Guzzo ~\cite{Guzzo94} and Guti\'errez Fern\'andez~\cite{Fernand00}. Further, we establish  the Peirce decomposition of  nonassociative algebras of cubic minimal cones (the so-called Hsiang algebras) and study its basic properties. Interestingly, the latter class can be thought as a generalization of the Nourigat--Varro algebras considered in  \cite{NourigatI}, \cite{NourigatII}, \cite{NourigatIII} in the setting of baric algebras.

%\begin{remark}
%We primarily work with identities in one variable like \eqref{paai} or \eqref{JaiBernstein}. The general case ... by substitution
%\end{remark}
\medskip
\textit{Acknowledgements.}
While some of our main results were obtained earlier during the  work on Hsiang algebras, the inspiration for this paper derives from the Axial Algebra Workshop at Bristol, May 2018. The author thanks the Heilbronn Institute and the organizers Justin McInroy and Sergey Shpectorov for making it possible to attend this event. The author also thanks  the participants of the workshop, in particular Yoav Segev and Tom de Medts for many fruitful discussions.

\section{Preliminaries ad some motivating examples}\label{sec:prelim}
In this section we  briefly recall some standard terminology that will be suitably extended to the labeled binary trees in the next sections, and also discuss several motivating examples.

By an algebra $A$ we shall always mean a commutative, maybe nonassociative, finite dimensional algebra over a  filed $\Field$ of $\mathrm{char}(\Field) \ne 2,3$.  An element $c$ of algebra $A$ is called an idempotent if $c^2=c$. Given an idempotent $c$, one can define the (left$=$right) multiplication endomorphism $L_c\in \mathrm{End}_\Field(A)$ by
$$
L_c:x\to cx.
$$
The  characteristic polynomial of $L_c$  is called \textit{Peirce
polynomial} of $c$. The set $\sigma(c)$ of the roots of the  characteristic polynomial  is called the \textit{Peirce spectrum} of the idempotent $c$. The Peirce spectrum is always nonempty because $1\in \sigma(c)$. An idempotent $c$ is called \textit{semi-simple} if $A$ splits into a direct sum of simple invariant submodules $A_c(\lambda)$, $\lambda\in \sigma(c)$:
\begin{equation}\label{Peircede}
A=\bigoplus_{\lambda\in \sigma_f(c)}A_c(\lambda),
\end{equation}
where $L_c$ acts as the multiplication by $\lambda$ on each $L_c$-invariant submodule $A_c(\lambda)$.
This decomposition is also known as the \textit{Peirce decomposition} of $A$ relative to $c$.

An important ingredient of the Peirce method is the multiplication structure of the $L_c$-invariant submodules. This structure is determined by the so-called \textit{fusion laws} prescribing how the product  $A_c(\lambda)\,A_c(\mu)$ decomposes in \eqref{Peircede} for various $\lambda,\mu\in \sigma(c)$. In other words, a fusion law is a map
$$
\star:\sigma(c)\times \sigma(c)\to 2^{\sigma(c)}
$$
 such that
\begin{equation}\label{fusionlaws}
A_c(\lambda)A_c(\mu)\subset \bigoplus_{\nu\in \lambda\star \mu}A_c(\nu).
\end{equation}
Sometimes it is convenient to assume that \eqref{fusionlaws} is minimal in an obvious sense.

For example, the Peirce polynomial associated with the Jordan algebra identity \eqref{Jai2} does not depend on a particular choice of an idempotent $c$ and given by
\begin{equation}\label{charJai}
\varrho_c(P,L_c):=2L_c^3-3L_c^2+L_c=2(L_c-\half12)(L_c-1)L_c,
\end{equation}
see \cite{Schafer} and also the derivation of \eqref{charJai} in Example~\ref{ex:Jordan} below. Therefore the Peirce  spectrum  consists of three eigenvalues: $\sigma_c(P)=\{1,0,\half12\}$. A further argument based on the second order linearization  reveals  that the $L_c$-invariant submodules satisfy \eqref{halfpeirce} the \textit{Jordan  fusion laws}:
\begin{equation}
	\renewcommand{\arraystretch}{1.3}
	\setlength{\tabcolsep}{0.65em}
	\begin{array}{c|ccc}
			\star\,\,	& 1\,\, &\,\, 0 & \,\, \half12  \\
		\hline
			 1  &  \{1\}  &\,\,  \emptyset  &\,\,  \{\half12\}  \\
			 0  &	&\,\,  \{0\}  &\,\,  \{\half12\}  \\
			 \half12  &	&	&\,\,  \{1,0\}  \\
	\end{array}
	\end{equation}

In general, the linearization technique is an important tool in extracting the Peirce decomposition and the corresponding fusion laws in a  nonassociative algebra with an identity  \cite{Zhevlakov}, \cite{McCrbook}, \cite{Rowenbook}. We discuss the linearization technique and the related  concepts for general nonassociative structures in Sections~\ref{sec:trees}--\ref{sec:second} below. In short, the Peirce spectrum $\sigma(P,c)$ of the identity $P$  emerges from the first order linearization $D^1(P;c,y)$, while the corresponding fusion rules are systematically extracted from the second order derivation $D^2(P;c,x,y)$. The latter technique requires some more care, but the key ingredient is simple: one can show that the second linearization evaluated at an idempotent always depends on the product of $xy$ but not on $x$ or $y$ individually. The latter readily yields fusion rules \eqref{fusionlaws}.
%We consider these concepts more detailed in Sections~\ref{sec:trees} -- \ref{sec:second} below.

It  is also worth mentioning  that the algebra identity concept  considered in this paper is different from the standard PI (polynomial identity) definition for  associative rings, cf. for example \cite[Sec.~23]{Rowenbook}. Our approach is rather in the spirit of \cite{Osborn}, \cite{McCr65} and \cite{Koecher} which especially  suitable for nonassociative structures in the presence of certain analytic structures. Still, the algebra identity in the sense of \cite{Osborn} is defined as a nonassociative polynomial identity with \textit{constant} coefficients, as for instance the power-associativity identity
\begin{equation}\label{paai}
zz^3-z^2z^2=0.
\end{equation}

Instead, we  allow the coefficients depend on the indeterminate to comprise all reasonable identities. These include, for example, all train baric algebras \cite{Reed} carrying  a nontrivial $\Ring$-homomorphism $\,\omega(x):A\to \Ring$ with a prominent example being the Bernstein algebras  satisfying
\begin{equation}\label{JaiBernstein}
z^2z^2=\,\omega(z)^2z^2,
\end{equation}
see \cite{WBusekros}, \cite{Lyubich92} and also Section~\ref{sec:appl} below for more detailed discussion of various classes of baric algebras.
Another well-known examples are pseudo-composition algebras \cite{Meyberg1} satisfying
\begin{equation}\label{PsComp}
z^3=b(z,z)\,z,
\end{equation}
with $b$ being a symmetric bilinear form, or general rank three algebras \cite{Walcher1},
\begin{equation}\label{Walch}
z^3=a(z) z^2+b(z)\,z.
\end{equation}
Some further examples include baric train algebras of general rank \cite{Reed}, \cite{Fernand00} and rank four identities considered recently in  \cite{NourigatI}--\cite{NourigatIII}.
We also mention the class of nonassociative commutative algebras of cubic minimal cones (the so-called Hsiang algebras) satisfying defining identity
\begin{equation}\label{Hsiang}
4zz^3+z^2z^2=3b(z,z)\,z^2+ 2b(z^2,z)\,z,
\end{equation}
where $b$ is an associating symmetric  bilinear form (i.e. satisfies \eqref{Qass} below). We discuss  this algebras and their Peirce decomposition in Section~\ref{sec:Hsiang} below.

All the above examples share a remarkable common property that the quantity $\frac12$ is in their Peirce spectrum. This value is exceptional in several respects, see a recent discussion in \cite{KrTk18a}, \cite{HRS15b}, \cite{HSS18}, \cite{HSS18b}, \cite{Tk18a}, \cite{Tk18b}. In particular, it was shown in \cite{Tk15b}, \cite{KrTk18a} that any idempotent of a commutative nonassociative algebra over reals with positive definite associative bilinear form is primitive and each its Peirce eigenvalue is $\le \half12$. It is the main purposes of this paper to establish a universal character of $\half12$ and the corresponding  fusion laws for an \textit{arbitrary} nonassociative algebra with an identity.

%In the present paper we establish the following general result.
%
%\begin{theorem}\label{th1}
%Let $A$ be a commutative nonassociative algebra satisfying a weighted  polynomial identity $f$  and let $c$ be a nonzero idempotent in $A$. Then $\frac12\in\sigma_f$.
%If additionally $c$ is a semi-simple idempotent and $\frac12$ is simple  root of the Peirce polynomial $\bar{\pi}_f(t)$ then \eqref{Vc12} holds.
%\end{theorem}

\section{Nonassociative algebras and complete binary trees}\label{sec:trees}

Below we consider the terminology for the general case of linearization of arbitrary order, although we are primarily interested in and make use of linearization of orders $1$ and $2$. Since the linearization method depends crucially on the relation between the degree of a defining identity and the characteristic of  the ground field $\Field$, some further care is needed when working with linearizations of order $\ge3$.  Note, however, that the requirement $\mathrm{char} \Field \ne2,3$ suffices for many of our prinipal results.

\subsection{Commutative groupoids}
Let us recall some standard concepts and definitions following \cite[Appendix 21B]{Rowenbook}, see also \cite{Kurosh47}, \cite{Osborn}. Let $N(X)$ be a commutative multiplicative groupoid generated by elements $x$ of an at most countable set $X$ (free magma). We shall denote the multiplicative operation by juxtaposition. In other words, $N(X)$ consists of all words of finite length that can be formed using the elements of $X$ and using parentheses to indicate the way in which each word is built up by a sequence of juxtapositions. Two words are considered distinct elements of $N(X)$ unless they are identical in every way including the positions of all the parentheses.

Since we consider the commutative case only we have
$$
x^3:=x(xx)=xx^2=x^2x, \qquad \text{where } x^2:=xx.
$$
Then the simplest example is the commutative groupoid generated by a single element $x$:
$$
N(x)=N(\{x\})=\{x^\alpha: x,\,\, x^2,\,\, x^3,\,\, x^2x^2,\,\, x^4,\,\, x^5,\,\, x(x^2x^2),\ldots\}
$$
Its elements are called nonassociative monomials.

The total degree $\deg x$ of an element $x\in N(X)$ is the number of elements of $X$ used in the word $x$ counting multiplicities. Thus, elements of $X$ have degree $1$, and the degree of a product of two elements of $N(X)$ is the sum of the degrees of the factors. Thus defined, the degree is one more than the number of \textit{products} needed to express the element in terms of elements of $X$, and two more than the number of pairs of parentheses needed to indicate the order in which the products are to be taken (if the degree is greater than 1). Furthermore, if $x_i \in X$ and $x\in N(X)$, the degree of $x_i$ in $x$ is the number of times that $x_i$ occurs
in the word representing $x$.

Let $\Ring$ be an commutative associative ring with unity element $1$ and let $\Ring(X)$ denote the free nonassociative algebra on $X$ over $\Ring$ (the magma algebra), i.e. the left $\Ring$-module given by the set of all finite linear combinations of elements of $N(X)$ with coefficients from $\Ring$  and  multiplication defined by
$$
(\sum\alpha_i z_i)(\sum\beta_j w_j)=\sum \alpha_i\beta_j (z_iw_j)
$$
According to \cite{Rowenbook}, an element of $\Ring(X)$ is called a nonassociative polynomial.

Then the above definition has a universal meaning  as the following elementary  observation shows.

\begin{proposition}[Proposition~1.1 in \cite{Osborn}]\label{pro:subs}
Let $A$ be any algebra over $\Ring$ and let $\,\varrho$ be a map
of the set $X$ into $A$. Then there exists exactly one $\Ring$-homomorphism $\theta$ of $\Ring(X)$ into $A$ such that $\theta(x)=\,\varrho(x)$ for all $x\in X$.
\end{proposition}

A map $\,\varrho$ in Proposition~\ref{pro:subs} is the substitution homomorphism of elements of an algebra $A$ for elements of $X$. An element
$f\in \Ring(X)$  is an \textit{identity} on a $\Ring$-algebra $A$ (or $A$ satisfies the identity $f$) if every $\Ring$-homomorphism of $\Ring(X)$ into $A$ takes $f$ into zero. In this paper, however, we consider a more general definition. It is more convenient from now on to work with a field $\Field$ instead of a ring $\Ring$. To avoid some obvious complifications, we shall always assume that
$$
\mathrm{char} \Field \ne 2,3,
$$
and also that $A$ is torsion-free, i.e. $\lambda x=0$ with $\lambda\in \Ring$ and $0\ne x\in A$ implies $\alpha=0$.
A function $\phi:A\to \Ring$ is called  a polynomial map of degree $m\in \{0,1,2,\ldots\}$ if for any fixed $x,y\in A$, $\phi(x+ty)$ is a polynomial of degree $m$ in $t$.

\begin{definition}
A commutative algebra $A$ is said to satisfy a \textit{weighted  identity}
%\footnote{Formally, one should write $\phi_{x^\alpha}$ instead of $\phi_{x^\alpha}$, but in practice one  has some specific coefficients labeled in a different way, so we prefer to abuse notation and write for short $\phi_{x^\alpha}$.}
\begin{equation}\label{identity}
P(z):=\sum_{z^\alpha\in N(X)}\phi_{z^\alpha}(z)z^\alpha =0,
\end{equation}
if the sum contains finitely many terms, each coefficient $\phi_{z^\alpha}(z)$ is  polynomial map, and \eqref{identity} holds for any substitution $\,\varrho:X\to A$.
\end{definition}

Setting in \eqref{identity} $\,\varrho(z)=c$, where $c\in A$ is an arbitrary nonzero idempotent, yields by virtue of
$$
c^\alpha=c
$$
that
$$
0=\sum_{\alpha}\phi_{z^\alpha}(c)c^\alpha=(\sum_{\alpha}\phi_{z^\alpha}(c))c.
$$
Since $A$ is torsion free and $c\ne 0$,  we have
\begin{equation}\label{xc}
\sum_{\alpha}\phi_{x^\alpha}(c)=0.
\end{equation}
Observe that this identity may be trivial, as for example for  the power-associative algebra identity \eqref{paai}. In general, however, this implies a constraint on the values of $\phi_{x^\alpha}$ evaluated for $c$.

\subsection{Complete binary trees and nonassociative monomials }
The simplest nonassociative monomials are the principal powers, defined by induction as
$$
z^1=z,\quad z^{n+1}=zz^n, \quad n\ge1.
$$
On the other hand, to work with general nonassociative monomials and, especially, to linearize them much more care is needed. To deal with the general case we interpret nonassociative monomials as complete binary trees. This formalism is very natural to work with the indices of powers in algebraic systems having non-associative multiplication and comes back to Etherington's papers in genetic algebras \cite{Etherington49}, \cite{Etherington62}. Etherington himself, however, applied never this formalism for recovering of the Peirce structure and worked primarily with combinatorial structure of a nonassociative multiplication. In this paper, we extend the binary tree formalism to labeled binary trees and develop it to study the general Peirce decomposition. Some  related results for baric algebras and for principal or plenary powers were recently obtained in \cite{Guzzo94}, \cite{Fernand00} and \cite{Mallol17}.

We recall the basic concepts and definitions of binary tress below, see for example \cite{Audibert}. A complete binary tree is a special case of a direct graph, but it more constructive to use the following recursive definition.

A complete (rooted) binary tree $T$ is defined as a nonempty finite set of elements, called \textit{nodes}, with a distinguished node $r(T)$, called the tree \textit{root},  such that either (i) $T$ consists of a single root (the \textit{trivial tree}), or (ii) $T$ contains except for the root an unordered pair of disjoint binary trees $T_1$ and $T_2$. Two complete binary trees are equivalent (or equal) if they are isomorphic as directed graphs.

In what follows by a slight abuse of terminology we say a `binary tree' instead of  a `complete binary tree'.

According to the above definition, a binary tree is a rooted tree $T$ where each of the nodes has either no successor,  or two successors,  in each case the nodes are said to have degree 0 or 2. Nodes of degree 0 are exactly the nodes with no children; they are called \textit{leaves}. Every node (excluding a root) in a tree is connected by a directed edge from exactly one other node; this node is called a parent.  The topmost node in the tree is  the tree root.
We denote
$$
\partial T=\{\text{the set of leaves of $T$}\}.
$$
%Nodes which are not leaves are called {internal nodes}. Nodes with the same parent are called siblings.
Then the number of leaves is called the \textit{degree of a binary tree}, denoted by $\deg T$:
$$
\deg T=\mathrm{card}\,\partial T.
$$
The trivial tree is the only binary tree of degree $1$.

According to the definition, any \textit{nontrivial}  binary tree consists of the tree root $r(T)$ and an unordered pair of two binary trees $T_1$ and $T_2$. We write this as follows:
$$
T=T_1\circ T_2=T_2\circ T_1.
$$
 In the converse direction, given two binary trees $T_1$ and $T_2$, let $T_1\circ T_2$ denote the binary tree obtained by joining of $T_1$ and $T_2$ with a new root being a parent to the two roots of $T_1$ and $T_2$.

A tree $T'$ is  a subtree of $T$, or $T'\le T$, if either $T'=T$ or there exists a sequence $(T_{i})_{0\le i\le 2k}$ with $T_0=T$ and $T_{2k}=T'$ such that
$$
T_0=T_1\circ T_2, \quad T_2=T_3\circ T_4, \ldots, \quad T_{2k-2}=T_{2k-1}\circ T_{2k}.
$$

Thus defined product  is obviously commutative. Any binary tree can be written as a nonassociative monomial of tree's leaves. It is easy to see that
\begin{equation}\label{productT}
\partial(T_1\circ T_2)=\partial T_1\sqcup \partial T_2
\end{equation}

A \textit{labeling} on a binary tree  $T$ is a  map
$$
f:\partial T\to X,
$$
where $X$ is an arbitrary set. Two labelings  are equal if they are equal as maps. A \textit{labeled tree }is a pair $(T,f)$, where $f$ is a labeling. Sometimes we write $(T,f)$ as $T$ if $f$ is clear from the context. If $T'\le T$ then the restriction $f|_{T'}$ is defined in an obvious way.

The simplest labeling is a \textit{constant labeling}:
$$
i_x:\partial T\to x, \quad x\in X.
$$

A constant labeling is a particular case of a \textit{dichotomic labeling}, i.e. a labeling $f$ such that $$
f(\partial T)\subset \{x,y\},
$$
where $x,y$ are arbitrary (maybe equal) elements of $X$.
The cardinality of the preimage $f^{-1}(x)$ is called the \textit{multiplicity} of the dichotomic labeling at $x$. Some examples of dichotomic labelings are given  in Fig.~\ref{fig:linear}.

\begin{definition}\label{def:pi}
Given a labeled tree $T=(T,f)$, the \textit{root product value}
$$
\pi:T\to N(X)
$$
is the nonassociative monomial of the free magma $X$ uniquely determined by
\begin{itemize}
\item
if $T$ is trivial then $\pi(T,f)=f(r(T))$,
\item
if $T=T_1\circ T_2$ then $\pi(T,f)=\pi(T_1,f|_{T_1})\,\pi(T_2,f|_{T_2})$.
\end{itemize}
\end{definition}

Some examples of the root product values are given in diagrams a) and b) in Fig.~\ref{fig:rootvalue}.
\begin{figure}

\begin{tikzpicture}[level distance=1cm,
  level 1/.style={sibling distance=2cm},
  level 2/.style={sibling distance=1cm}]
  \node[] {$(x_1x_2)(x_3x_4)$}
    child {node {$x_1x_2$}
        child {node {$x_1$}}
        child {node {$x_2$}}
    }
     child {node {$x_3x_4$}
        child {node (A) {$x_3$}}
        child {node {$x_4$}}
     };
   \node[yshift=5mm,below= of A]  {a)};
   %\node[below= of 00] (01) {}
\end{tikzpicture}
\qquad
\begin{tikzpicture}[level distance=1cm,
  level 1/.style={sibling distance=2cm},
  level 2/.style={sibling distance=1cm}]
  \node[style={circle}] {$x^2x^2$}
    child {node {$x^2$}
        child {node {$x$}}
        child {node {$x$}}
    }
     child {node {$x^2$}
        child {node (B) {$x$}}
        child {node {$x$}}
     };

     \node[yshift=5mm,below= of B]  {b)};

    %child {node {$w$}

\end{tikzpicture}
\qquad
\begin{tikzpicture}[level distance=1cm,
  level 1/.style={sibling distance=2cm},
  level 2/.style={sibling distance=1cm}]
  \node[] {$q^2(p_1+p_2)+qp_3$}
    child {node {$q(p_1+p_2)$}
      child {node {$p_1$}}
      child {node(C) { $p_2$}}
    }
    child {node{$p_3$}
    };
     \node[yshift=5mm,below= of C]  {c)};
\end{tikzpicture}
\caption{ }\label{fig:rootvalue}
\end{figure}

If $i_x$ is a constant labeling then
\begin{equation}\label{holds1}
\pi(T,i_x)=x^\alpha\in N(\{x\}).
\end{equation}
In the converse direction, given an element $x^\alpha\in N(\{x\})$ there exists a unique complete binary tree $T$ such that \eqref{holds1} holds; we denote it by $\Tree{z^\alpha}$. In other words,
\begin{equation}\label{xxx}
\pi(\Tree{z^\alpha},i_x)=x^\alpha.
\end{equation}
In this setting, the number of leaves of the monomial tree $\Tree{z^\alpha}$ is exactly the degree of the monomial:
\begin{equation}\label{degeq}
\deg \Tree{z^\alpha}=\deg z^\alpha.
\end{equation}

An elementary but important corollary of \eqref{xxx} is that for any idempotent $c$
\begin{equation}\label{ccc}
\pi(\Tree{z^\alpha},i_c)=c^\alpha=c.
\end{equation}

The symmetric group $\mathrm{Sym}(\partial T)$ of permutations of the tree leaves acts naturally on labelings:
$$
f^\sigma(t):=f\circ \sigma(t), \quad t\in \partial T.
$$
The \textit{total root product value} is the element of the free commutative nonassociative algebra $\Field (X)$ obtained by summing up all possible permutations of a given labeling, i.e.
$$
\bar{\pi}(T,f)=\sum_{\sigma\in \mathrm{Sym}(\partial T)}\pi(T,f^\sigma)\in \Field(X).
$$

For example,
\begin{equation}\label{m!}
\bar{\pi}(\Tree{z^\alpha}, i_x)=m!\, x^\alpha, \qquad m=\deg z^\alpha.
\end{equation}

The above terminology becomes clear from the following key observation which proof  is an easy corollary of the definitions.

\begin{proposition}
 \label{pro:linearization}
Let $A$ be a commutative nonassociative algebra over $\Field$, and let  $x,y\in A$. Let $z^\alpha\in N(\{z\})$ be an nonassociative monomial, $m=\deg z^\alpha$. Then for any $t\in \Field$
\begin{equation}\label{sumlinear}
(x+ty)^\alpha=x^\alpha+tD(z^\alpha;x,y)+ t^2D^2(z^\alpha;x,y)+\ldots+ t^mD^m(z^\alpha;x,y),
\end{equation}
where
\begin{equation}\label{Ddef}
D^k(z^\alpha;x,y)=\frac{1}{k!(m-k)!}\bar{\pi}(\Tree{z^\alpha},f_k),
\end{equation}
and $f_k:\partial \Tree{z^\alpha}\to (x,y)$ is a dichotomic labeling of multiplicity $k$ at $y$.
\end{proposition}

Thus defined  expression $D^k(z^\alpha;x,y)$ is called the linearization of $z^\alpha$ of order $k$ evaluated at $x$ along $y$. In the associative setting, the linearization coincides with the classical directional derivative of the monomial $z^\alpha=z^m$.

\begin{figure}

\begin{tikzpicture}[level distance=1cm,
  level 1/.style={sibling distance=2cm},
  level 2/.style={sibling distance=1cm}]
  \node[style={circle}] {$x(xy)$}
    child {node {$yx$}
      child {node {$y$}}
      child {node {$x$}}
    }
    child {node {$x$}
    };
\end{tikzpicture}\qquad\quad
\begin{tikzpicture}[level distance=1cm,
  level 1/.style={sibling distance=2cm},
  level 2/.style={sibling distance=1cm}]
  \node[style={circle}] {$x(xy)$}
    child {node {$yx$}
      child {node {$x$}}
      child {node {$y$}}
    }
    child {node {$x$}
    };
\end{tikzpicture}
\qquad\quad
\begin{tikzpicture}[level distance=1cm,
  level 1/.style={sibling distance=2cm},
  level 2/.style={sibling distance=1cm}]
  \node[style={circle}] {$yx^2$}
    child {node {$x^2$}
      child {node {$x$}}
      child {node {$x$}}
    }
    child {node {$y$}
    };
\end{tikzpicture}
\caption{The first order linearization terms of $\Tree{x^3}$}
\label{fig:linear}
\end{figure}

\begin{example}\label{ex1}
Let us consider the linearization of the nonassociative monomial $z^3$. From the algebraic point of view we have
$$
(x+t y)^3= x^3+(x^2y+2x(xy))t+(xy^2+2y(yx))t^2+y^3t^3,
$$
therefore we find by identification
\begin{align}
D^1(z^3;x,y)&=x^2y+2x(xy)=(L_{x^2}+2L_x^2)y,\label{llll}\\
D^2(z^3;x,y)&=xy^2+2y(yx),\nonumber\\
D^3(z^3;x,y)&=y^3\nonumber.
%D^3(x^3;y)&=2x(yz)+2z(yx)+2y(zx).
\end{align}
Alternatively, the linearization of the first order can be obtained  by considering the binary tree
$$
\Tree{x^3}=\Tree{x}\circ \Tree{x^2}=\Tree{x}\circ (\Tree{x}\circ \Tree{x})
$$
of degree $3=\deg x^3$, where $\Tree{x}$ is a trivial tree labeled by $x$. Then summing up the elementary terms obtained by `relabeling' of each $x$ by $y$ yields
$$
\Tree{y}\circ (\Tree{x}\circ \Tree{x})+
\Tree{y}\circ (\Tree{y}\circ \Tree{x})+
\Tree{x}\circ (\Tree{x}\circ \Tree{y}),
$$
where $+$ here should be understood as the addition in the corresponding free magma algebra on $X=\{\Tree{x},\Tree{y}\}$. See also  the  binary tree diagrams in Figure~\ref{fig:linear}.

\end{example}

Combining \eqref{sumlinear} with \eqref{m!} yields an analogue of the Euler homogeneous function theorem holds:
\begin{equation}\label{Eulerx}
D^k(z^\alpha;x,x)=\frac{\bar{\pi}(\Tree{z^\alpha},i_x)}{k!(m-k)!}=\binom{m}{k} x^\alpha.
\end{equation}

\begin{proposition}
  There holds
\begin{equation}\label{Ddef1}
D^k(z^\alpha;x,y):=\sum_g\pi(\Tree{z^\alpha},g),
\end{equation}
where the sum is taken over all $\binom{m}{k}$ distinct dichotomic labelings $g:\partial \Tree{z^\alpha}\to (x,y)$ of multiplicity $k$ at $y$.
\end{proposition}

\begin{proof}
For there are exactly $\binom{m}{k}$ distinct dichotomic labelings of multiplicity $k$ at $y$, hence \eqref{Ddef} implies \eqref{Ddef1}.
\end{proof}

If $k=1$ then $y\to D^1(x^\alpha;y)$ is a linear endomorphism of the algebra $A$ written as a sum of certain compositions of operators $L_{x^\beta}$, see \eqref{llll} and some explicit representations the second row in Table~\ref{tab1}.

Finally, we mention that the \textit{full linearization} of $x^\alpha$ is defined by
$$
D(z^\alpha;y_1,\ldots,y_m)=\bar{\pi}(\Tree{z^\alpha},f),
$$
where $f$ is a bijective labeling $f:\partial T\to \{y_1,\ldots,y_m\}$, where $m=\deg z^\alpha$. Then $D(z^\alpha;y_1,\ldots,y_m)$ is symmetric $m$-linear form in $y$'s:
$$
D(z^\alpha;y_{i_1},\ldots,y_{i_m})=D(z^\alpha;y_1,\ldots,y_m)
$$
for any permutation $i\in \mathrm{Sym}(1,2,\ldots,m)$. It follows from \eqref{Ddef} that
\begin{equation}\label{Dxy}
D(z^\alpha;\underbrace{x,\ldots,x}_{m-k},\underbrace{y,\ldots,y}_{k})=
k!(m-k)!D^k(z^\alpha;x,y).
\end{equation}

\subsection{Linearizations of polynomial maps}\label{sec:linear}

A map  $\phi(x):A\to \Field$ on an algebra over $\Field$ is called a polynomial map if
\begin{equation}\label{linear}
\phi(x+ty)=\phi(x)+\sum_{i=1}^{m}t^i\delta_i\phi(x;y)
\end{equation}
is a polynomial in $t\in \Field$ of a fixed degree (independent on $t$).
Then each coefficient $\delta_i\phi(x;y)$ is a polynomial map of any of the variables $x$ and $y$; it is homogeneous of degree in $m-i$ in $x$ and degree $i$ in $y$, respectively. A simple argument reveals that the Euler homogeneous function theorem holds, i.e. the substitution $y=x$ in the linearization restores the polynomial $\phi$ (up to a constant factor independent of $\phi$):
\begin{equation}\label{Euler}
\delta_i\phi(x;x)=\binom{m}{i}\phi(x).
\end{equation}

 The polynomial map $y\to \delta_k\phi(x;y)$ is called the \textit{linearization} of $\phi(x)$ of order $k$.

As before, we shall primarily concern ourselves with the linearizations of order 1 and 2. The second coefficient in \eqref{linear}, $\delta_2\phi(x;y)$, is quadratic with respect to $y$, thus can further be polarized to get a bilinear form $\delta_2\phi(x;y,z)$ in $y,z$ defined by
\begin{equation}\label{delta2}
\delta_2\phi(x;y,z)=\delta_2\phi(x;y+z)-\delta_2\phi(x;y)-\delta_2\phi(x;z)
\end{equation}
such that the Euler homogeneous function theorem yields
\begin{equation}\label{linear2}
\delta_2\phi(x;y,y)=2\delta_2\phi(x;y).
\end{equation}
%Two binary trees are equivalent if the corresponding root values are equal.

\section{The Peirce operator} \label{sec:Peirce}
%The linearization technique is an important tool in study nonassociative algebras. We start with linearizations of order one which are responsible for the algebra spectrum. In short, the idea of the method is as follows: starting with an algebra identity $P=0$ one linearizes it near an idempotent $c\in A$. This yields a polynomial equation on the multiplication (adjoint) operator $L_c$. The roots of the obtained polynomial over the ground field $\Field$ are the Peirce numbers of the idempotent $c$ and they form the Peirce spectrum $\sigma(c)$. If an idempotent $c$ is simple then $A$ decomposes into a direct sum of $L_c$-invariant submodules. In other words, $L_c$ acts as a dilatation $x\to \lambda x$ on the $L_c$-invariant submodule $A_\lambda$. An important ingredient of the Peirce method is the multiplication structure of thus obtained $L_c$-invariant submodules, i.e.
%$$
%A_\lambda A_\mu\subset \bigoplus_{\nu\in \lambda\star \mu}A_\nu,
%$$
%where the correspondence $$
%(\lambda,\mu)\to \lambda\star \mu\subset \sigma(c)
%$$ is a so-called fusion law of $A$ relative to $c$ \cite{HRS15}. One wants to make the subset $\lambda\star \mu$ as small as possible to get a finer algebraic structure. On practice, $\sigma(c)$ emerges from the first order linearization $D^1(P;c,y)$, while the corresponding fusion rules can be extracted from the second order derivation $D^2(P;c,x,y)$. We consider this in more detail below.

We begin with the monomial case $P=z^\alpha$ and then proceed with a general $P$ in the next sections.

\begin{definition}\label{def:omega}
Let $\Modul$ be a left module over $\Field$. Given  $q\in \Field$ and  a labeled binary tree $(T,f)$, where $f:\partial T\to \Modul$, we define the \textit{Peirce operator}
$$
\,\varrho:(T,f)\to \,\varrho(T,f,q),
$$
uniquely determined by the following conditions:
\begin{itemize}
\item
if $T$ is trivial then $\,\varrho(T,f,q)=f(r(T))$,
\item
if $T=T_1\circ T_2$ then $\,\varrho(T,f,q)=q(\,\varrho(T_1,f|_{T_1},q)+\,\varrho(T_2,f|_{T_2},q))$.
\end{itemize}
The total Peirce operator is defined by
$$
\Omega(T,f,q)=\sum_{\sigma\in \mathrm{Sym}(\partial T)}\,\varrho(T,f^\sigma,q).
$$
\end{definition}

See an explicit example of the Peirce operator for $T=\Tree{z^3}$  in the diagram c), Fig.~\ref{fig:rootvalue}.

In what follows, we always consider the case  $\Modul=\Field$. Then  $\,\varrho(T,q)$ is a polynomial in $q$ with coefficients in $\mathbb{Z}[f(\partial T)]$. An important  specialization is when
$f=i_1$ is the constant labeling sending all leaves of $T$ to the field unit $1$, namely:
$$
\,\varrho(T,q):=\,\varrho(T, i_1,q)\in \mathbb{Z}[q].
$$

If $T$ is a monomial tree, the polynomial $\,\varrho(\Tree{z^\alpha},q)$ is called the \textit{Peirce polynomial} of $z^\alpha$ and denoted for short by
$$
\,\varrho(z^\alpha,q):=\,\varrho(\Tree{z^\alpha},q).
$$
In particular,
\begin{equation}\label{omega1}
\begin{split}
\,\varrho (z,q)&=1,\\
\,\varrho (z^2,q)&=2q,\\
\,\varrho (z^3,q)&=2q^2+q,
\end{split}
\end{equation}
and in general
\begin{equation}\label{log1}
\,\varrho(z^\alpha z^\beta,q)=q\bigl(\,\varrho(z^\alpha,q)+\,\varrho(z^\beta,q)\bigr)
\end{equation}
See also some explicit examples in Table~\ref{tab1}. The following identities  will be used during the proof of the main results.

\begin{proposition}\label{pro:half}
For any monomial $z^\alpha$ there holds:
\begin{align}
\varrho(z^\alpha,1)&=\deg z^\alpha,\nonumber\\
\varrho(z^\alpha,\half12)&=1, \nonumber\\
\varrho(z^\alpha,0)&=
\left\{
  \begin{array}{ll}
   1&\deg z^\alpha=1\\
0&\deg z^\alpha\ge 2.
  \end{array}
\right.\nonumber
\end{align}
\end{proposition}

\begin{proof}
The proof of all identities is by induction by the degree $m=\deg z^\alpha$. When $m=1$ the statement is trivial. Assume that the statement is true for any $k$, $1\le k\le m$, where $m\ge2$. Consider an arbitrary monomial  $z^\alpha$ of degree $m+1$ and decompose it in a proper product $z^\beta z^\gamma$. Then by \eqref{log1},
$$
\,\varrho(z^\beta z^\gamma,q)=q\bigl(\,\varrho(z^\beta,q)+\,\varrho(z^\gamma,q)\bigr)
$$
hence we have by the induction assumptions
\begin{align*}
\,\varrho(z^\beta z^\gamma,1)&=1\cdot (\,\varrho(z^\beta,1)+\,\varrho(z^\gamma,1))=\deg z^\beta+\deg z^\gamma=\deg z^\alpha,\\
\,\varrho(z^\beta z^\gamma,\half12)&=\half12\cdot \left(\,\varrho(z^\beta,\half12)+\,\varrho(z^\gamma,\half12)\right) =\half12(1+1)=1,\\
\,\varrho(z^\beta z^\gamma,0)&=0\cdot(\,\varrho(z^\beta,1)+\,\varrho(z^\gamma,1))=0,
\end{align*}
the proposition follows.
\end{proof}

\begin{remark}
 It is an easy corollary of \eqref{log1} that  nonzero coefficients of $\,\varrho(z^\alpha,q)$ are always positive and are powers of $2$. Proposition~\ref{pro:half}  shows that the Peirce polynomials share  nice properties of $z^\alpha$. It would be interesting to understand the algebraic properties of Peirce polynomials better.
\end{remark}

Our next step is to connect the Peirce operator $\,\varrho$ to the first order linearizations. First note that the operator $y\to D^1(z^\alpha;c,y )$ is a linear endomorphism of $A.$ We show that it is actually a polynomial in $L_c$. The identity \eqref{PeirceD1} below takes the central place in the Peirce theory, stating that the linearization operator $D^1(z^\alpha;c,\cdot)$ evaluated at an idempotent $c$ has the same eigenvectors as $L_c$ and its eigenvalues can be explicitly evaluated by the Peirce polynomial at the corresponding eigenvalues of $L_c$.

\begin{proposition}
Let $c$ be an idempotent of an algebra $A$ over $\Field$. Then
\begin{equation}\label{PeirceD}
D^1(z^\alpha;c,y )=\,\varrho(z^\alpha,L_c)\,y.
\end{equation}
In particular, if $y\in A$ is such that $cy=\lambda y$, $\lambda\in \Field$, then
\begin{equation}\label{PeirceD1}
D^1(z^\alpha;c,y )=\,\varrho(z^\alpha,\lambda)\,y.
\end{equation}
\end{proposition}

\begin{proof}
The proof is by induction by $m=\deg z^\alpha$. If $m=1$ then a  dichotomic labeling of order $1$ in \eqref{Ddef} is constant: $f_1=i_y$, hence
$$
D^1(z;c,y )=\pi(\Tree{z},f_1)=\pi(\Tree{z},i_y)=i_y(r(T))=y,
$$
and on the other hand by \eqref{omega1},
$$
\,\varrho(z,q)=1 \quad \Rightarrow \quad \,\varrho(z,L_c)y=y,
$$
hence \eqref{PeirceD} follows.
Next, assume that \eqref{PeirceD} is valid for all monomials of degree less or equal  $m-1\ge 1$. Consider $z^\alpha$ of degree $m$ and write it as a proper product $z^\alpha=z^\beta z^\gamma$, where $\deg z^\beta=k<m $ and $\deg z^\gamma=m-k<m$. Then by the induction assumption
\begin{equation}\label{indass1}
\begin{split}
 D^1(z^\beta;c,y )&=\,\varrho(z^\beta,L_c)\,y,\\
D^1(z^\gamma;c,y )&=\,\varrho(z^\gamma,L_c)\,y.
\end{split}
\end{equation}
On the other hand,  from \eqref{Ddef1}  and Definition~\ref{def:pi} we have
\begin{align*}
D^1(z^\alpha;c,y)&=\sum_{g\in G_\alpha} \pi(\Tree{z^\alpha},g)=
\sum_{g\in G_\alpha} \pi(\Tree{z^\beta},g)\,\pi(\Tree{z^\gamma},g),
\end{align*}
where $G_\alpha$ is the set of $m=\binom{m}{1}$ distinct dichotomic labelings $g:\partial \Tree{z^\alpha}\to (x,y)$ of order $1$.
Since the preimage  $t:=g^{-1}(y)$ has cardinality one, we have a natural disjoint union
\begin{equation}\label{disjoint}
G_\alpha=G_\beta\sqcup G_\gamma,
\end{equation}
where $g\in G_\beta$ (resp. $g\in G_\gamma$) if and only if $t\in \partial \Tree{z^\beta}$ (resp. $t\in \partial \Tree{z^\gamma}$). If $g\in G_\beta$ then $g|_{G_\gamma}=i_c$ is constant on $\partial \Tree{z^\gamma}$, hence we obtain from \eqref{ccc}:
$$
\pi(\Tree{z^\gamma},g)=\pi(\Tree{x^\gamma},\,i_c)=c^\gamma=c,
$$
therefore
\begin{align*}
\pi(\Tree{z^\beta},g)\,\pi(\Tree{z^\gamma},g)&=
c\cdot \pi(\Tree{z^\beta},g)=L_c\pi(\Tree{z^\beta},g).
\end{align*}
Since \eqref{disjoint} we obtain by virtue of \eqref{Ddef1} and  \eqref{indass1}
\begin{align*}
D^1(z^\alpha;c,y)&=\sum_{g\in G_\beta}L_c \pi(\Tree{z^\beta},g) + \sum_{g\in G_\gamma}L_c \pi(\Tree{z^\gamma},g)\\
&=L_c D^1(z^\beta;c,y )+L_cD^1(z^\gamma;c,y )\\
&=L_c\bigl(\,\varrho(z^\beta,L_c)+\,\varrho(z^\gamma,L_c)\bigr)\,y\\
&=\,\varrho(z^\beta z^\gamma,L_c)y\\
&=\,\varrho(z^\alpha,L_c)y,
\end{align*}
where the last equalities are by virtue of \eqref{log1}. This proves \eqref{PeirceD} by induction. The proof of \eqref{PeirceD1} is obvious.
\end{proof}

The above computations were valid for the constant labeling $i_1$. For a general labeling $f$, it is important to know explicit expressions for the total Peirce polynomial $\Omega(T, f,q)$. The following proposition shows that the general case amounts to the constant one.
\begin{proposition}\label{pro:Pi}
\begin{equation}\label{PeircePi}
\Omega(T,f,q)=(m-1)!\,\trace(f)\,\,\varrho(T,q),
\end{equation}
where the trace of $f$ is defined by
$$
\trace( f)=\sum_{t\in \partial T}f(t).
$$
\end{proposition}

\begin{proof}
Indeed, it follows from Definition~\ref{def:omega} that $\,\varrho(T,q,f)$, and therefore also $\Omega(T,q,f)$, are linear in $f$. Furthermore, since the total Peirce polynomial $\Omega$ is obtained by symmetrization, one has
$$
\Omega(T,f,q)=C  \trace (f)
$$
for some $C\in \Field$ independent on $f$. To determine $C$ we consider $f=i_1$, then
$$
\Omega(T,i_1,q)=C  \trace( i_1)=C\sum_{t\in \partial T}1=C\deg T=mC.
$$
On the other hand, since $i_1$ is constant we have $\,\varrho(T,i_1^\sigma,q)=\,\varrho(T,i_1,q)$ for all $\sigma\in \mathrm{Sym}(\partial T)$, hence
$$
\Omega(T,i_1,q)=m!\,\,\varrho(T,i_1,q)=m!\, \,\varrho(T,q),
$$
therefore $C=(m-1)!\,\,\varrho(T,q)$, which finishes the proof.
\end{proof}

\section{Second order linearizations}\label{sec:second}
Now we derive explicit expressions for linearizations of the second order evaluated at an idempotent $c\in A$ along $x,y\in A$. We start with \eqref{Ddef1},
\begin{equation}\label{Ddef2}
D^2(z^\alpha;c,y):=\sum_{g}\pi(\Tree{z^\alpha},g),
\end{equation}
where the sum is taken over $\frac{m(m-1)}{2}$ distinct dichotomic labelings $g:\partial \Tree{z^\alpha}\to (c,y)$ of order $2$, $m=\deg z^\alpha$. The map $y\to D^2(z^\alpha;c,y)$ is quadratic and its polarization
\begin{equation}\label{delta1}
D^2(z^\alpha;c,x,y)=D^2(z^\alpha;c,x+y)-D^2(z^\alpha;c,x)-D^2(z^\alpha;c,y)
\end{equation}
is a bilinear form in $x,y$. An easy computation shows that $D^2(z;c,y)=0$ and $D^2(z^2;c,y)=y^2$,
hence
\begin{equation}\label{mm2}\begin{split}
D^2(z;c,x,y)&= 0, \\
D^2(z^2;c,x,y)&=2xy.
\end{split}
\end{equation}
In general, combining  \eqref{Ddef2} with \eqref{delta1} yields
\begin{equation}\label{Ddef3}
D^2(z^\alpha;c,x,y)=\sum_{h\in H(z^\alpha)}\pi(\Tree{z^\alpha},h),
\end{equation}
where $H(z^\alpha)$ denote the set of all distinct labelings $h:\partial \Tree{z^\alpha}\to \{c,x,y\}$, where value $c$ has the multiplicity two. The cardinality of $H(z^\alpha)$ is $m(m-1)$ and is nonzero for $m\ge2$.

Since the cases $m=1$ and $m=2$ are already known, see \eqref{mm2}, we may assume that $m\ge 3$. Under this assumption, let us consider a proper product $z^\alpha=z^\beta z^\gamma$, where $\deg z^\beta=k<m $ and $\deg z^\gamma=m-k<m$.
Rewrite \eqref{Ddef3} as
\begin{equation}\label{Ddef4}
D^2(z^\alpha;c,x,y)=\left(\sum_{\mathrm{a})} +\sum_{\mathrm{b})} +\sum_{\mathrm{c})} +\sum_{\mathrm{d})}\right)\pi(\Tree{z^\alpha},h),
\end{equation}
where $h^{-1}(x)$ and $h^{-1}(y)$ are distributed according to  the four possible cases displayed in the table below.
	\renewcommand{\arraystretch}{1.5}
	\setlength{\tabcolsep}{0.85em}

\begin{table}[h]
\begin{tabular}{c|c|c}

 & $\partial \Tree{z^\beta}$ & $\partial \Tree{z^\gamma}$\\\hline
a)& $h^{-1}(x),h^{-1}(y)$ & \\\hline
b)& $h^{-1}(x)$ & $h^{-1}(y)$  \\\hline
c)& $h^{-1}(y)$ & $h^{-1}(x)$  \\\hline
d)& & $h^{-1}(x),h^{-1}(y)$  \\
\end{tabular}
%\caption{}\label{tab3}
\end{table}

We have
$$
\pi(\Tree{z^\alpha},h)=\pi(\Tree{z^\beta},h|_{\Tree{z^\beta}})\, \pi(\Tree{z^\gamma},h|_{\Tree{z^\gamma}}).
$$

Then in  the case a) one has $h|_{\Tree{z^\beta}}\in H(z^\beta)$ and $h|_{\Tree{z^\gamma}}$ is a constant labeling:  $h|_{\Tree{z^\gamma}}=i_c$, hence
$$
\pi(\Tree{z^\alpha},h)=L_c\pi(\Tree{z^\beta},h|_{\Tree{z^\beta}}),
$$
therefore
$$
\sum_{\mathrm{a})} \pi(\Tree{z^\alpha},h)=
\sum_{h\in H(z^\beta)} L_c\pi(\Tree{z^\beta},h|_{\Tree{z^\beta}})=
L_c\,D^2(z^\beta;c,x,y).
$$
A similar argument yields
$$
\sum_{\mathrm{d})} \pi(\Tree{z^\alpha},h)=
L_c\,D^2(z^\gamma;c,x,y).
$$
Next, in case b) we have that both $h|_{\Tree{z^\beta}}$ and $h|_{\Tree{z^\gamma}}$ are order one dichotomy labelings $(c,x)$ and $(c,y)$ respectively: $h|_{\Tree{z^\beta}}=i_c$ and $h|_{\Tree{z^\gamma}}=i_c$. Hence using \eqref{PeirceD} we derive
$$
\sum_{\mathrm{b})} \pi(\Tree{z^\alpha},h)=
D^1(z^\beta;c,x ) D^1(z^\gamma;c,y )=(\,\varrho(z^\beta,L_c)x)(\,\varrho(z^\gamma,L_c)y).
$$
Similarly one obtains
$$
\sum_{\mathrm{c})} \pi(\Tree{z^\alpha},h)=
(\,\varrho(z^\beta,L_c)y)(\,\varrho(z^\gamma,L_c)x).
$$
In summary, we have from \eqref{Ddef4} the following recursion:
\begin{equation}\label{Ddef5}
\begin{split}
D^2(z^\alpha;c,x,y)&=L_c\bigl(D^2(z^\beta;c,x,y)+D^2(z^\gamma;c,x,y)\bigr)\\
&+(\,\varrho(z^\beta,L_c)x)(\,\varrho(z^\gamma,L_c)y) +(\,\varrho(z^\beta,L_c)y)(\,\varrho(z^\gamma,L_c)x)
\end{split}
\end{equation}

An important corollay of \eqref{Ddef5} is the following observation.

\begin{proposition}\label{th311}
If $x\in A_c(\lambda)$ and $y\in A_c(\mu)$ then there exists a linear endomorphism $\frak{D}(z^\alpha;\lambda,\mu,L_c)$ of $A$ such that
\begin{equation}\label{xyprod}
D^2(z^\alpha;c,x,y) =\frak{D}(z^\alpha;\lambda,\mu,L_c)xy.
\end{equation}
Here $\frak{D}$ is uniquely determined by
\begin{itemize}
\item[(i)]
$\frak{D}(z;\lambda,\mu,L_c)=0$ and
\item[(ii)]
$
\frak{D}(z^\beta z^\gamma;a,b,L_c)=L_c(\frak{D}(z^\beta;a,b,L_c)+\frak{D}( z^\gamma;a,b,L_c))
+\,\varrho(z^\beta,a)\,\varrho(z^\gamma,b) +\,\varrho(z^\beta,b)\,\varrho(z^\gamma,a).
$
\end{itemize}
\end{proposition}

\begin{proof}
First note that specializing $x\in A_c(\lambda)$ and $y\in A_c(\mu)$ in \eqref{Ddef5} we obtain
\begin{equation}\label{Ddef6}
\begin{split}
D^2(z^\alpha;c,x,y)&=L_c\bigl(D^2(z^\beta;c,x,y)+D^2(z^\gamma;c,x,y)\bigr)\\
&+(\,\varrho(z^\beta,\lambda)\,\varrho(z^\gamma,\mu) +\,\varrho(z^\beta,\mu)\,\varrho(z^\gamma,\lambda))xy.
\end{split}
\end{equation}
We proceed by induction on the degree of monomial $z^\alpha$. By virtue of \eqref{mm2}, the claim is true when the degree is $\le 2$. Assuming that the claim is valid for all monomials of degree $\le m$, where $m\ge 2$, we consider an arbitrary $z^\alpha$ of degree $m+1$ and decompose it into a proper product $z^\alpha=z^\beta z^\gamma$. Then \eqref{Ddef6} gives
\begin{equation*}
\begin{split}
D^2(z^\alpha;c,x,y)&=L_c\bigl(\frak{D}(z^\beta;\lambda,\mu,L_c)+ \frak{D}(z^\gamma;\lambda,\mu,L_c)\bigr)xy +(\,\varrho(z^\beta,\lambda)\,\varrho(z^\gamma,\mu) +\,\varrho(z^\beta,\mu)\,\varrho(z^\gamma,\lambda))xy,
\end{split}
\end{equation*}
thus implying the desired conclusion.
\end{proof}

We want to emphasize that it follows from \eqref{xyprod} that the second linearization $D^2$ depends only on the product of $x$ and $y$, but not the elements separately. It also follows from Proposition~\ref{th311} that $\frak{D}(z^\alpha;a,b,L_c)$ is a commutative associative polynomial  in $a,b,L_c$ with integer coefficients. It is convenient to associate to $\frak{D}(z^\alpha;a,b,L_c)$ the polynomial
$$
\frak{D}(z^\alpha;a,b,p)\in \mathbb{Z}[a,b,p],
$$
which will be referred to as the \textit{Peirce symbol} of the nonassociative monomial $z^\alpha$.

For general values of $a$ and $b$, an explicit form of the Peirce symbol is rather involved, see however some explicit formulae in Table~\ref{tab2}. Some special cases when it is possible a closed simple form for tthe Peirce symbol are discussed in propositions below.

\begin{proposition}
For any $a ,p\in \Field$ there holds
\begin{equation}\label{eH}
\frak{D}(z^\alpha;a,\half12,p)=\frac{\,\varrho(z^\alpha,p) -\,\varrho(z^\alpha,a)\vphantom{\,\varrho(z^\alpha,p)}}{p-a}.
\end{equation}
\end{proposition}

\begin{proof}
By Proposition~\ref{pro:half}, $\,\varrho(z^\alpha,\half12)=1$, hence  using (ii) in Proposition~\ref{th311}  we  find
$$
\frak{D}(z^\beta z^\gamma;a,\half12,p)= q(\frak{D}(z^\beta;a,\half12,p)+\frak{D}( z^\gamma;a,\half12,p)+\,\varrho(z^\alpha,a)+\,\varrho(z^\beta,a).
$$
Let us define $$g(z^\alpha,a,p)=(p-a)\frak{D}(z^\alpha;a,\half12,p)+\,\varrho(z^\alpha,a).
$$ Then applying the latter identity we obtain
\begin{align*}
g(z^\beta z^\gamma,a,p)&=
(p-a)\frak{D}(z^\beta z^\gamma;a,\half12,p)+\,\varrho(z^\beta z^\gamma,a)\\
&=
(p-a)\frak{D}(z^\beta z^\gamma;a,\half12,p)+p(\,\varrho(z^\beta,a)+\,\varrho( z^\gamma,a))\\
&=p(p-a)\bigl(\frak{D}(z^\beta;a,\half12,p)+\frak{D}( z^\gamma;a,\half12,p)\bigr)+q(\,\varrho(z^\beta,a)+\,\varrho( z^\gamma,a))\\
&=q(g(z^\beta,a,p)+ g(z^\gamma,a,p)).
\end{align*}
By (i) in Proposition~\ref{th311} we have $\frak{D}(z;a,\half12,p)=0$, hence $q(z,a,p)=1$. This implies by induction that $g(z^\alpha,a,p)$ does not depend on $a$ and also that  $g(z^\alpha,a,p)=\,\varrho(z^\alpha,p)$ for any $z^\alpha$, cf. Definition~\ref{def:omega}. The latter identity  yields \eqref{eH}.
\end{proof}

\begin{table}
\begin{tabular}{l|l|l|l}
$z^\alpha$ & $D(z^\alpha;x,\cdot)$ & $\,\varrho(z^\alpha,q)$ &  $\frak{D}(z^\alpha,\lambda,\mu)$ \\\hline
$z$ & $1$& $1$&$0$\\
$z^2$ & $2L_x$& $2q$&$2$\\
$z^3$ & $L_{x^2}+2L_x^2$& $2q^2+q$& $2(L_c+\lambda+\mu)$\\
$z^4$ & $L_{x^3}+L_xL_{x^2}+2L_x^3$& $2q^3+q^2+q$&$2L_c^2+2aL_c+2\lambda^2 +2\mu^2+\lambda+\mu-4\lambda\mu$\\
$z^2z^2$ & $4L_{x^2}L_x$& $4q^2$&$4L_c+8\lambda\mu$\\
%$x^5$ & $L_{x^4}+L_xL_{x^3}+L_x^2L_{x^2}+2L_x^4$& $2q^4+q^3+q^2+q$&$\ldots$\\
\end{tabular}
\smallskip
\caption{Some explicit formulae of the Peirce polynomials}\label{tab1}
\end{table}

\begin{proposition}\label{pro:principal}
For the principal powers
$$
z^1=z, \quad z^2=zz, \quad z^n=z^{n-1}z \quad\forall n\ge 2,
$$
there holds
\begin{align}\label{principalpowers1}
\,\varrho(z^{n},q)&=\frac{2q^n-q^{n-1}-q}{q-1}, \quad \forall n\ge 1\\
\label{principalpowers2}
\frak{D}(z^{n};a,b,p)&=R_n(p,a)+R_n(p,b)-R_n(p,\half12),
\end{align}
where
$$
R_n(x,y)=\frac{\varrho(z^n,x)-\varrho(z^n,y)}{x-y}.
$$

\end{proposition}

\begin{proof}
Since $\varrho(z,q)=1$ ae have
$$
\,\varrho(z^n,q)=\,\varrho(zz^{n-1},q)=q(\,\varrho(z^{n-1},q)+1),
$$
therefore \eqref{principalpowers1} easily follows by induction. Similarly, note that $\frak{D}(z^1;a,b,p)=0$ trivially satisfies \eqref{principalpowers2}. Arguing by induction we assume that \eqref{principalpowers2} holds for all principal powers $\le n$. By (ii) in Proposition~\ref{th311} and $\frak{D}( z;a,b,p)=0$  we obtain
\begin{equation}\label{cancel}
\begin{split}
\frak{D}(z^{n+1};a,b,p)&=p(\frak{D}(z^{n};a,b,p)+\frak{D}( z;a,b,p))
+\,\varrho(z^{n},a)\,\varrho(z,b) +\,\varrho(z^{n},b)\,\varrho(z,a)\\
&=p\frak{D}(z^{n};a,b,p)
+\,\varrho(z^{n},a) +\,\varrho(z^{n},b)\\
&=p\left(R_n(p,a)+R_n(p,b)-R_n(p,\half12)\right)+\,\varrho(z^{n},a) +\,\varrho(z^{n},b)
\end{split}
\end{equation}
Since
$$
R_{n+1}(x,y)=\frac{\varrho(z^{n+1},x)-\varrho(z^{n+1},y)}{x-y}=
1+\frac{x}{x-y}\varrho(z^n,x)+\frac{y}{y-x}\varrho(z^n,y),
$$
we have
\begin{align*}
R_{n+1}(x,y)-xR_n(x,y)&=
1+\frac{x}{x-y}\varrho(z^n,x)+\frac{y}{y-x}\varrho(z^n,y)- \frac{x}{x-y}\varrho(z^n,x)-\frac{x}{y-x}\varrho(z^n,y)\\
&=1+\varrho(z^n,y),
\end{align*}
therefore we find from \eqref{cancel} after cancelations
\begin{align*}
(R_{n+1}(p,a)+R_{n+1}(p,b)-R_{n+1}(p,\half12))-\frak{D}(z^{n+1};a,b,p)=
1-\varrho(z^{n},b)=0,
\end{align*}
thus implying \eqref{principalpowers2} by induction.
\end{proof}

\begin{proposition}\label{pro:plenarypow}
For the plenary powers
$$
z^{[1]}=z, \qquad z^{[n]}=z^{[n-1]}z^{[n-1]} \quad\forall n\ge 2,
$$
there holds for all $n\ge0$:
\begin{align}\label{plenarypowers1}
\,\varrho(z^{[n+1]},q)&=(2q)^n, \\
\label{plenarypowers2}
\frak{D}(z^{[n+1]};a,b,p)&=\frac{\varrho(z^{[n]},p)- \varrho(z^{[n]},a)\,\varrho(z^{[n]},b)}{p-2ab}=
2^n\,\frac{p^n- (2ab)^n}{p-2ab}.
\end{align}

\end{proposition}

\begin{proof}
The proof is by induction and is left to the reader. We mention only the induction step in \eqref{plenarypowers2}:
\begin{align*}
\frak{D}(z^{[n+2]};a,b,p)&=2\biggl(p\frak{D}(z^{[n+1]};a,b,p)+ \varrho(z^{[n+1]},a)\,\varrho(z^{[n+1]},b)\biggr)\\
&=2^{n+2}\left(\frac{p^{n+2}- p\cdot (2ab)^{n+1}}{p-2ab} +  (2ab)^{n+1}\right)\\
&=2^{n+2}\frac{p^{n+2}-(2ab)^{n+2}}{p-2ab}.
\end{align*}
\end{proof}

\section{The Peirce polynomial and the Peirce symbol}\label{sec:Peircepol}
Now we assume that an algebra $A$ satisfies a weighted identity
\begin{equation}\label{Identity}
P(z):=\sum_{\alpha}\phi_{z^\alpha}(z)z^\alpha =0,
\end{equation}
where only finitely many terms are nonzero.  Specializing $z=c$, where $c^2=c$ is an algebra idempotent, gives
\begin{equation}\label{xc1}
\sum_{\alpha}\phi_{z^\alpha}(c)=0.
\end{equation}
 The linearization of \eqref{Identity} at $x\in A$ along $y\in A$ yields a new identity
\begin{equation}\label{linear1st}
\sum_{\alpha}\delta_1\phi_{z^\alpha}(x,y)x^\alpha+\phi_{z^\alpha}(x)D^1(z^\alpha;x,y) =0
\end{equation}
which is linear in $y$. Specialization $x=c$ gives by virtue of \eqref{PeirceD}
\begin{equation}\label{PPei}
(\sum_{\alpha}\delta_1\phi_{z^\alpha}(c,y))c+ \bigl(\sum_{\alpha}\phi_{z^\alpha}(c)\,\varrho(z^\alpha,L_c)\bigr)y=0
\end{equation}

\begin{definition}
The polynomial
\begin{equation}\label{def:peirce}
\,\varrho_c(P,t):=\sum_{\alpha}\phi_{z^\alpha}(c)\,\varrho(z^\alpha,t)
\end{equation}
is called the \textit{Peirce
polynomial} of the algebra identity \eqref{identity} at the idempotent $c$. The set of zeros of $\,\varrho_c(P,q)$ is called the \textit{Peirce spectrum of the identity} $P$ relative to the idempotent $c$, denoted by $\sigma(P,c)$. As a convention we write $\sigma(P,c)=\Field$ when $\varrho_c(P)\equiv 0$. A value $\lambda\in \sigma(P,c)$ is called \textit{simple} if $\lambda$ is a single root of $\varrho_c(P,t).$
\end{definition}

An algebra $A$ with identity $P$ is called \textit{homogeneous} if the Peirce polynomial does not depend on a choice of  idempotent $c$. In that case we write just $\varrho(P,t).$ The simplest example is any algebra satisfying an identity with constant coefficients is homogeneous. But there many   examples of homogeneous algebras satisfying non-constant coefficient identities. These include all baric algebras and Hsiang algebras (see below).

\begin{example}\label{ex:Jordan}
We illustrate the above concepts for the well-known class of Jordan algebras. We start with \eqref{Jai2} and substitute $y=x$ to obtain
\begin{equation}\label{powerasso}
P=xx^3-x^2x^2=0.
\end{equation}
It is well-known that the latter identity implies that $A$ must be power-associative when $\mathrm{char} \Field\ne 2,3,5$ \cite{Albert48}.
In our notation, \eqref{powerasso} has constant coefficients: $\phi_{xx^3}=1$ and $\phi_{x^2x^2}=-1$. Therefore the resulting Peirce polynomial does not depend on a choice of an idempotent and may be found using Table~\ref{tab1}:
\begin{equation}\label{powerP}
\,\varrho(P,q)=\,\varrho(zz^3,q)-\,\varrho(z^2z^2,q)=2q^3+q^2+q-4q^2=q(2q-1)(q-1)
\end{equation}
which implies that the Peirce spectrum of a Jordan algebra does not depend on an idempotent and equal $\{0,\frac12,1\}$. In fact, as we know, for a concrete idempotent $c$,  the spectrum of $L_c$ may vary and the multiplicity of  $\frac12$ maybe zero for certain idempotents. It is well-known \cite{FKbook}, \cite{Koecher} that if $A$ is a formally real Jordan algebra then its simplicity is equivalent to the presence of $\frac12$ in the spectrum of any idempotent.
\end{example}

Using this definition, \eqref{PPei} amounts to the following condition:
\begin{align}
\label{ppee}
 \,\varrho_c(P,L_c)y&=-\tau(y) c,\qquad\text{where}\quad\tau(y):=\sum_{\alpha}\delta_1\phi_{z^\alpha}(c,y)\in \Field.
\end{align}
Thus, the Peirce polynomial of $P$ of $L_c$ is (at most) rank one endomorphism of $A$. It also follows from \eqref{ppee} that if $y\nparallel c$ is an eigenvector of $L_c$ with eigenvalue $\lambda$ then
$$
y\in \ker \,\varrho_c(P,L_c)\cap \ker \tau.
$$
In particular, this implies

\begin{proposition}
 \label{pro:lamm}
If $y\nparallel c$ is an eigenvector of $L_c$ with eigenvalue $\lambda$ then
\begin{align}\label{ylam}
 \,\varrho_c(P,\lambda)&=0, \\
\tau(y)&=0,\label{ylam1}
\end{align}
i.e.  $\lambda$ is in the Peirce spectrum $\sigma(P,c)$:
$$
\sigma(c)=\sigma(L_c)\subset \sigma(P,c)\cup\{1\}.
$$
\end{proposition}

\begin{remark}
We should emphasize that the Peirce spectrum of the identity  $\sigma(P,c)$ may or may not contain the eigenvalue $1$ depending on $P$ or $c$. But the Peirce spectrum $\sigma(c)=\sigma(L_c)$ of the idempotent $c$ always contain $1$ because of $cc=c$. In other words, if $1\not\in \sigma(P,c)$ then $c$ is primitive in the sense that $\lambda=1$ has multiplicity one and the span $A_c(1)=\R{}c $ is one-dimensional.
\end{remark}

%Thus, given an idempotent $c$ of an algebra $A$ with an identity $P$, any eigenvalue of $L_c$ (except may be $1$) must be a root of the Peirce polynomial $\,\varrho_c(P,q)$.

\begin{proposition}
 \label{pro:D2}
Let $c$ be a semi-simple idempotent  and let \begin{equation}\label{directsum}
A=\bigoplus_{\nu\in \sigma(c)}A_c(\nu),
\end{equation}
be the corresponding direct sum decomposition of $A$ into invariant submodules $A_c(\lambda_k)$ of $L_c$ corresponding to the eigenvalues $\nu\in \sigma(P,c)$, i.e. $L_c=\nu$ on $A_c(\nu)$. Then for any $x\in A_c(\lambda)$ and $y\in A_c(\mu)$
\begin{equation}\label{Qeq}
\sum_{\nu\in \sigma(c)} Y(\lambda,\mu,\nu)(xy)_{\nu}=
-\tau_2(x,y)c-\tau_1(x,\mu)y-\tau_1(y,\lambda)x,
\end{equation}
where for arbitrary $u,v\in A$ and $a\in \Field$
\begin{align*}
Y(\lambda,\mu,\nu)&= Y(\varrho_c(P),\lambda,\mu,\nu)=\sum_{\alpha}
\phi_{z^\alpha}(c)\frak{D}(z^\alpha,\lambda,\mu,\nu),\\
\tau_1(u,a)&=\tau_1(\varrho_c(P),u,a)= \sum_{\alpha}\delta_1\phi_{z^\alpha}(c;u)\,\varrho(z^\alpha,a),\\
\tau_2(u,v)&=\tau_2(\varrho_c(P),u,v)=\sum_{\alpha}\delta_2\phi_{z^\alpha}(c;u,v),
\end{align*}
and $xy=\sum_{\nu}(xy)_{\nu}$ is the decomposition according to \eqref{directsum}.
\end{proposition}

\begin{proof}
Linearizing further \eqref{linear1st}, we obtain
\begin{equation*}\label{linear2nd}
\sum_{\alpha}\left(\delta_2\phi_{z^\alpha}(x;y)x^\alpha+ \delta_1\phi_{z^\alpha}(x;y)D^1(z^\alpha;x,y)+ \phi_{z^\alpha}(x)D^2(z^\alpha;x,y)\right) =0
\end{equation*}
The substitution $x=c$ yields a quadratic in $y$ identity
\begin{equation*}
\bigl(\sum_{\alpha}\delta_2\phi_{z^\alpha}(c;y)\bigr)c+ \sum_{\alpha}\delta_1\phi_{z^\alpha}(c;y)D^1(z^\alpha;c,y)+ \phi_{z^\alpha}(c)D^2(z^\alpha;c,y) =0
\end{equation*}
Polarizing the latter identity using \eqref{delta1},  \eqref{delta2} and \eqref{PeirceD} yields
\begin{align*}
\bigl(\sum_{\alpha}\delta_2\phi_{z^\alpha}(c;x,y)\bigr)c&+ \bigl(\sum_{\alpha}\delta_1\phi_{z^\alpha}(c;x)\,\varrho(z^\alpha,L_c)\bigr)y
+
\bigl(\sum_{\alpha}\delta_1\phi_{z^\alpha}(c;y)\,\varrho(z^\alpha,L_c)\bigr)x\\
&+\sum_{\alpha}
\phi_{z^\alpha}(c)D^2(z^\alpha;c,x,y) =0.
\end{align*}
Now suppose that $x\in A_c(\lambda)$ and $y\in A_c(\mu)$. Then $\,\varrho(z^\alpha,L_c)x=\,\varrho(z^\alpha,\lambda)x$ and $\,\varrho(z^\alpha,L_c)y=\,\varrho(z^\alpha,\mu)y$, hence using \eqref{def:peirce} and \eqref{xyprod} we obtain \eqref{Qeq}, the proposition follows.
\end{proof}

\begin{definition}
The polynomial $Y(\varrho_c(P),\lambda,\mu,\nu)$ is called the \textit{Peirce symbol of the identity} $P$ relative to the idempotent $c$.
\end{definition}

The left hand side of \eqref{Qeq} is a linear combination of $(xy)_\nu\in A_c(\nu)$ for distinct $\nu$, and the right hand side contains at most three terms (depending on the coefficients $\tau$'s). This implies several observations which are important for extracting of fusion laws.

\begin{corollary}[Fusion laws]
\label{cor:fusionlaws}
Under assumptions of Proposition~\ref{pro:D2},
    $$
  Y(\lambda,\mu,\nu)\,(xy)_\nu=0\qquad \text{for any $\nu\not\in\{1,\lambda,\mu\}$.}
  $$
  Further, if $\nu\in\{1,\lambda,\mu\}$ and $Y(\lambda,\mu,\nu)\ne0$ then (up to the permutation $\lambda \leftrightarrow\mu$) the following three cases are possible:
\begin{itemize}
  \item[(i)]
  $\nu=1$: if $Y(\lambda,\mu,1)\ne0$ then
  $$
  (xy)_1\in \Span( c,\{x,y\}\cap A_c(1)).
  $$
  \item[(ii)]
 $\nu=\lambda\ne1$ and $\lambda\ne \mu$: if $Y(\lambda,\mu,\lambda)\ne0$ then $$Y(\lambda,\mu,\lambda)(xy)_\lambda=-\tau_1(y,\lambda)x.
 $$
   \item[(iii)]
 $\nu=\lambda\ne1$ and $\lambda= \mu$: if $Y(\lambda,\lambda,\lambda)\ne0$ then $$Y(\lambda,\lambda,\lambda)(xy)_\lambda=-\tau_1(y,\lambda)x- \tau_1(x,\lambda)y.
 $$
\end{itemize}
\end{corollary}

\section{Proof of Theorem~\ref{th1}}
With the  results of the previous sections in hand, we are  ready to finish the proof of the theorem.
Let an algebra $A$ satisfy \eqref{identity} and let $c\ne0$ be an algebra idempotent. By Proposition~\ref{pro:half}, $\,\varrho(z^\alpha,\frac12)=1$ for any monomial $z^\alpha$, hence using \eqref{xc1} we obtain
\begin{equation}\label{Pc12}
\,\varrho_c(P,\half12)=\sum_{\alpha}\phi_{z^\alpha}(c)\,\varrho(z^\alpha,\half12) =
\sum_{\alpha}\phi_{z^\alpha}(c),
\end{equation}
hence
$$
\,\varrho_c(P,\half12)=0,
$$
which shows that $\half12\in \sigma(P,c)$ and therefore yields the first part of the theorem.

% It follows from Proposition~\ref{pro:lamm} that if $\nu\ne1$ then $\nu\in \sigma(P,c)$.
In notation of Proposition~\ref{pro:D2} we obtain  by virtue of Proposition~\ref{pro:half} and \eqref{ylam1} for any $u\in A_c(\nu)$, $\nu\ne1$ that
\begin{equation}\label{tau0}
\tau_1(u,\half12)=
\sum_{\alpha}\delta_1\phi_{z^\alpha}(c;u)\,\varrho(z^\alpha,\frac12)=
\sum_{\alpha}\delta_1\phi_{z^\alpha}(c;u)=\tau(u)=0.
\end{equation}
Hence, \eqref{Qeq} yields by virtue of Proposition~\ref{pro:half} and using the definition of $\tau$ in \eqref{ppee}  that
\begin{equation}\label{Qeq1}
\begin{split}
\sum_{\nu\in \sigma(P,c)} Y(\lambda,\half12,\nu)(xy)_{\nu}
=-\tau_2(x,y) c- \tau_1(y,\lambda)x,\qquad \forall x\in A_c(\lambda), y\in A_c(\half12).
\end{split}
\end{equation}
Next, note that using \eqref{eH}
\begin{align*}
Y(\lambda,\half12,\nu)&=
\sum_{\alpha} \phi_{z^\alpha}(c)\frak{D}(z^\alpha;\lambda,\half12,\nu )
=\sum_{\alpha} \phi_{z^\alpha}(c)\frac{\,\varrho(z^\alpha,\nu ) -\,\varrho(z^\alpha,\lambda)}{\nu -\lambda}
=\frac{\,\varrho_c(P,\nu ) -\,\varrho_c(P,\lambda)}{\nu -\lambda},
\end{align*}
hence, since $\nu $ and  $\lambda$ are zeros of $\,\varrho_c(P,t)$ we have
\begin{equation}\label{Y12lam}
Y(\lambda,\half12,\nu)=
\left\{
  \begin{array}{ll}
   0&\nu\ne \lambda\\
\varrho_c'(P,\lambda)&\nu=\lambda.
  \end{array}
\right.
\end{equation}
Thus, \eqref{Qeq1} can be rewritten as
\begin{equation}\label{fus1}
\varrho_c'(P,\lambda)(xy)_{\lambda}
=-\tau_2(x,y) c- \tau_1(y,\lambda)x.
\end{equation}
By the assumption of the theorem the root $\lambda$ is simple, hence $\,\varrho_c'(P,\lambda)\ne0$.

If $\lambda=1$ then $x\in A_c(1)$ and therefore \eqref{fus1} implies $(xy)_1\in \Span( c,x)\subset A_c(1)$.

If $\lambda=\frac12$ then by \eqref{tau0} $\tau_1(y,\lambda)=\tau_1(y,\frac12)=0$, hence \eqref{fus1} yields $$(xy)_{\frac12}=-\tau_2(x,y) c,$$
implying $(xy)_{\frac12}=0$ and $\tau_2(x,y)=0$. In particular,
$$
A_c(\half12)A_c(\half12)\subset \bigoplus_{\nu\ne \frac12}A_c(\nu).
$$
which proves the claim (B) in Theorem~\ref{th1} for $\lambda=\frac12$.

Finally, if $\lambda\ne1,\frac12$ then \eqref{fus1} yields $$\varrho_c'(P,\lambda)(xy)_{\lambda}
=- \tau_1(y,\lambda)x,\qquad \tau_2(x,y)=0.
$$
This implies $(xy)_{\lambda}=0 $ and finishes the proof of Theorem~\ref{th1}.

\section{Degenerated identities}\label{sec:trivial}

According to Proposition~\ref{pro:lamm}, the spectrum of any idempotent in an algebra with a nontrivial identity $P(z)=0$ is a subset of the zero locus of the corresponding  Peirce polynomial $\,\varrho_c(P,q)$. This implies a nontrivial \textit{a priori} information about the spectrum as soon as the Peirce polynomial is not identically zero. But it may happen that for some idempotent $c$ there holds $\,\varrho_c(P,q)\equiv 0$. In that case one cannot derive any nontrivial information about the spectrum of $L_c$ directly from the algebra identity. This motivates the following definition.

\begin{definition}
An identity $P=0$ is called \textit{degenerated} relative to an idempotent $c$ if its Perce polynomial $\,\varrho_c(P,q)$ is identically zero.
\end{definition}

\begin{example}
One of the simplest degenerated identities is the baric algebra satisfying
\begin{equation}\label{ElLab}
P(z):=z^2z^2-2\omega(z)z^3+\omega(z)^2z^2=0,
\end{equation}
where $a:A\to \Field$ is an algebra homomorphism. Indeed, since $\omega(c)=1$ for all nonzero idempotents $c$, we have by virtue of Table~\ref{tab1}
$$
\,\varrho_c(P,q)=4q^2-2(2q^2+q)+2q\equiv 0.
$$
Algebras satisfying \eqref{ElLab} has been appeared in \cite{Bayara10} and studied by Elduque and Labra in \cite{EldLabra} (where it appears as identity (4)). It was proven in \cite{EldLabra} that the `gametization' of the original multiplication in $A$ to $x\star y = xy-\frac12\omega(x)y-\frac12\omega(y)x$ satisfies the plenary nilpotent identity $x^{\star2}\star x^{\star2}=0$, while $A^{\star2}\ne A$. The converse is also true; see Section~5 in \cite{EldLabra} for further information and open questions. Similar questions in the setting of baric or train algebras were studied in \cite{Mallol15}.

\end{example}

Note that in the above example the polynomial identity is factorizable:
$$
P=(z^2-a(z)z)^2.
$$
This motivates a natural question: what can be said about the Peirce spectrum of the algebras satisfying a \textit{decomposable} identity, i.e. when $P$ can be factorized as a proper product of two nonconstant nonassociative  polynomials
\begin{equation}\label{P1P2}
P(z)=P_1(z)P_2(z).
\end{equation}
The next proposition completely characterize the Peirce spectrum of decomposable identities: it turns out that the Peirce spectrum of a decomposable identity is either undetermined or  coincides (up to $\lambda=0$) with the Peirce spectrum of one of the factors.

\begin{proposition}
\label{pro:P1P2}
Let $A$ be an algebra with identity $P(z)$ such that $P$ is decomposable in the free nonassociative algebra $\Field(\{z\})$. Then for any nonzero idempotent $c$ of $A$ there holds
$$
\,\varrho_c(P_1,\half12)\,\,\varrho_c(P_2,\half12)=0.
$$
Furthermore,
\begin{itemize}
\item
if $\,\varrho_c(P_1,\frac12)=\,\varrho_c(P_2,\frac12)=0$ then $P$ is a degenerated identity;
\item
if $\,\varrho_c(P_i,\frac12)=0$ and $\,\varrho_c(P_j,\frac12)\ne0$ then $\sigma_c(P)=\{0\}\cup \sigma_c(P_i)$.
\end{itemize}
In particular, if $P=aP_1^k$, $a\in \Field$ and $k\ge2$ is an integer,  then the identity $P$ is degenerated.
\end{proposition}

\begin{proof}
Let us consider
$$
P_1=\sum_\alpha \phi_{z^\alpha}z^\alpha, \quad
P_2=\sum_\beta \psi_{z^\beta}z^\beta.
$$
Then
$$
P=\sum_\alpha\sum_\beta \phi_{z^\alpha}\psi_{z^\beta}z^\alpha z^\beta=:
\sum_\gamma \Phi_{z^\gamma}(z) z^\gamma.
$$
Then specializing $z=c$ readily yields a scalar identity
$$
\,\varrho_c(P,\half12)=\sum_\gamma \Phi_{z^\gamma}(c)=0,
$$
cf. \eqref{Pc12}, which implies
\begin{equation}\label{Pcc}
\,\varrho_c(P_1P_2,\half12)=\,\varrho_c(P_1,\half12)\,\,\varrho_c(P_2,\half12)=0.
\end{equation}

Furthermore, using \eqref{log1} we find
\begin{align*}
\,\varrho_c(P,q)&=\sum_\alpha\sum_\beta \phi_{z^\alpha}(c)\psi_{z^\beta}(c)\,\varrho(z^\alpha z^\beta,q)\\
&=q\,\sum_\alpha\sum_\beta \phi_{z^\alpha}(c)\psi_{z^\beta}(c)(\,\varrho(z^\alpha,q)+\,\varrho( z^\beta,q))\\
&=q\sum_\beta \psi_{z^\beta}(c)\sum_\alpha \phi_{z^\alpha}(c)\,\varrho(z^\alpha,q)+q\,\sum_\alpha \psi_{z^\alpha}(c)\sum_\beta \phi_{z^\beta}(c)\,\varrho(z^\beta,q)
\end{align*}
which yields
\begin{equation}\label{log2}
\,\varrho_c(P_1P_2,q)=
q\,\biggl(\,\varrho_c(P_2,\half12)\,\,\varrho_c(P_1,q) +
\,\varrho_c(P_1,\half12)\,\,\varrho_c(P_2,q)\biggr).
\end{equation}

Since $\Field$ is a field, it follows from \eqref{Pcc} that at least one of $\,\varrho_c(P_1,\frac12)$ and $\,\varrho_c(P_2,\frac12)$ must be zero.  If $\,\varrho_c(P_1,\frac12)=\,\varrho_c(P_2,\frac12)=0$ then $\,\varrho_c(P,q)\equiv0$, therefore $P$ is a degenerated identity. Next, let $\,\varrho_c(P_1,\frac12)=0$ and $\,\varrho_c(P_2,\frac12)\ne0$. Then \eqref{log2} yields
$$
\,\varrho_c(P_1P_2,q)=q\,\,\varrho_c(P_2,\half12)\,\,\varrho_c(P_1,q).
$$
Therefore the Peirce spectrum of $\sigma_c(P)=\{0\}\cup \sigma_c(P_1)$.

Finally, if $P=aP^k$ then we may choose $P_1=P$ and $P_2=aP^{k-1}$. It follows from \eqref{Pcc}  by induction that $\,\varrho_c(P,\frac12)=0$, hence  we get from \eqref{log2} that $\,\varrho_c(P,q)\equiv 0$. This finishes the proof of the proposition.
\end{proof}

\begin{remark}
The identity \eqref{log2} is interesting by its own right. In particular,
taking into account that $\,\varrho(z^\alpha,\frac12)=1$ for any monomial $z^\alpha$ by Proposition~\ref{pro:half}, the identity \eqref{log2} can be thought of as a generalization of \eqref{log1} on general nonassociative polynomials.
%b) Note also that for $q=1$ \eqref{log2} becomes the Leibnitz product rule:
%\begin{equation}\label{log3}
%\,\varrho_c(P_1P_2,1)=\,\varrho_c(P_2,\half12)\,\,\varrho_c(P_1,1) +
%\,\varrho_c(P_1,\half12)\,\,\varrho_c(P_2,1).
%\end{equation}
%i.e. $P\to \,\varrho_c(P,1)$ is a derivation of the algebra of all identities.
\end{remark}

%Let
%\begin{equation}\label{Identity0}
%P(z):=\sum_{\alpha}\phi_{z^\alpha}(z)z^\alpha =0,
%\end{equation}
%be a trivial identity relative to an idempotent $c\in A$. Then applying Proposition~\ref{pro:half} one deduce some a priori identities obtained by specialization $q=\half12, 1$ and $0$ respectively:
%\begin{align}
%&\sum_{\alpha}\phi_{z^\alpha}(c)=0,\label{triv1}\\
%&\sum_{\alpha}\phi_{z^\alpha}(c)\deg z^\alpha=0,\label{triv2}\\
%&\phi_{0}(c)=0,
%\end{align}
%where $\phi_0$ is the coefficient of $z$.

\section{Some applications to baric algebras}\label{sec:appl}

%\subsection{Baric algebras}
An algebra $A$ is called \textit{baric} if it admits a non-trivial algebra homomorphism $\omega:A\to \Field$ into the ground field $\Field$.
Here we revisit some well-established classes of baric algebras which appears in symbolic genetic.
A summary of recent results on baric algebras and basic references on algebras in genetics  can be found in Reed \cite{Reed} and W\"orz-Busekros (up to 1980) \cite{WBusekros}.

The homomorphism $\omega$ is called the  weight homomorphism. It easily follows from the definition that the kernel $\ker \omega$ is an ideal of $A$ and also that $\omega(c)=0$ or $\omega(c)=1$ for any idempotent $c$ of $A$. If $c\ne0$ is an idempotent in a train algebra then $\omega (e) = 1$, see Lemma~4.2 in \cite{WBusekros}.

Note that if $x\in A_c(\lambda)$  then
$$
\omega(cx)=\omega(\lambda x)=\lambda\omega(x),
$$
and on the other hand, $\omega(cx)=\omega(c)\omega(x)$, thus, $\lambda\ne1$ immediately implies that $x\in \ker \omega$. In particular,
\begin{equation}\label{keromega}
\bigoplus_{\lambda\ne 1} A_c(\lambda)\subset \ker \omega.
\end{equation}

There are two important well-studied classes of baric algebras: principal and plenary train algebras which have been introduced by Etherington \cite{Etherington} in connection with the symbolism of genetics, see also the references A. Worz-Busekros \cite{WBusekros}, Yu. Lyubich \cite{Lyubich92} and M. L. Reed  \cite{Reed} for more information.

\subsection{Principal train algebras}
Let $A$ be a baric train algebra of rank $n\ge2$, i.e. a commutative algebra satisfying
\begin{equation}\label{train1}
\gamma_0z^n+\gamma_1\omega(z)z^{n-1}+\ldots+\gamma_{n-1}\omega(z)^{n-1}z=0,
\end{equation}
where $\gamma_0=1$. In our notation, we have
\begin{equation}\label{omegaphi}
\phi_{z^k}(z)=\gamma_{n-k}\omega^{n-k}(z).
\end{equation}
In particular, the latter implies that $\omega(c)=1$ for any algebra idempotent $c\ne0$. Thus, the Peirce polynomial of an algebra satisfying \eqref{train1} does not depend on a choice of an idempotent. Note that substitution of $c$ in \eqref{train1} implies
\begin{equation}\label{condrho}
\sum_{k=0}^{n-1} \gamma_k=0
\end{equation}
Hence, using \eqref{principalpowers1} we find
\begin{align*}
\,\varrho(P,t)=\,\varrho_c(P,t)&=\sum_{k=1}^{n} \gamma_{n-k}\,\varrho(z^{k},t)=
\sum_{k=1}^{n} \gamma_{n-k}\frac{2t^k-t^{k-1}-t}{t-1}=(2t-1)T(t),
\end{align*}
where
$
T(t):=\frac{1}{t-1}\sum_{k=1}^{n} \gamma_{n-k}t^{k-1}
$
is a polynomial (note that the numerator vanishes at $t=1$ by virtue of \eqref{condrho}). Further, using \eqref{keromega} we obtains for any $x\in A_c(\lambda)$, $\lambda\ne1$,
\begin{align*}
\delta_1\omega^k(z)(c,x)&=k\omega^{k-1}(c)\omega(x)=0,\\
\delta_2\omega^k(z)(c,x,y)&=\frac{k(k-1)}{2}\omega^{k-2}(c)\omega(x)\omega(y)=0
\end{align*}
Therefore, it follows  for any $x\in A_c(\lambda)$ and $y\in A_c(\mu)$ by virtue of \eqref{Qeq}, \eqref{omegaphi} and \eqref{principalpowers2} that
\begin{equation}\label{Yplpl}
Y(\lambda,\mu,\nu)=\frac{\varrho(P,\nu)-\varrho(P,\lambda)}{\nu-\lambda} +\frac{\varrho(P,\nu)-\varrho(P,\mu)}{\nu-\mu}- \frac{\varrho(P,\nu)-\varrho(P,\frac12)}{\nu-\frac12}.
\end{equation}
This yields
$$
Y(\lambda,\mu,\nu)=Y(\mu,\lambda,\nu)=
\left\{
  \begin{array}{ll}
   0&\nu\ne \lambda,\mu,\half12 \,\,\text{ or }\,\,\nu=\lambda=\half12\ne \mu\\
\varrho'(P,\lambda)&\nu=\lambda\ne \mu,\half12\\
2\varrho'(P,\lambda)&\nu=\lambda=\mu\ne\half12\\
\varrho'(P,\half12)&\nu=\lambda=\mu=\half12\\
  \end{array}
\right.
$$
For example, if the Peirce spectrum $\sigma(P,c)$ is simple (i.e. all roots of $\varrho(P,t)$ are single), then one easily obtains the fusion laws found  by Guzzo, see Theorem~3.2 in \cite{Guzzo94}.

\subsection{Plenary train algebras}
Now let $A$ be a plenary train algebra of rank $n\ge2$, i.e. a commutative algebra satisfying the plenary identity
\begin{equation}\label{plenary1}
P:=h_0z^{[n]}+h_1(z)z^{[n-1]}+\ldots+ h_{n-2}z^{[2]}+ h_{n-1}z^{[1]}=0
\end{equation}
where $h_{k}=\gamma_{k}\omega^{2^{n-1}-2^{n-k-1}}(z)$ and $h_0=\gamma_0=1$. The simplest and prominent example is the so-called Bernstein algebras \cite{Etherington}, \cite{Holgate}, \cite{WBusekros} satisfying
\begin{equation}\label{JaiBernstein1}
x^2x^2=\,\omega(x)^2x^2.
\end{equation}
Though the algebraic structure of general Bernstein algebras is well-studied, their complete classification remains still a difficult unsolved problem.

We apply our method to the Peirce decomposition of a general plenary train algebra. As before, we have for any nonzero idempotent of $A$ that $\omega(c)=1$. Therefore, using Proposition~\ref{pro:plenarypow} one finds the Peirce polynomial
\begin{align*}
\varrho(P,t)=\varrho_c(P,t)&=\sum_{k=1}^{n} h_{n-k}(c)\,\varrho(z^{[k]},t)=
\sum_{k=1}^{n} (2t)^k\gamma_{n-k}.
\end{align*}
The latter Peirce polynomial is also known as the \textit{plenary train polynomial} of $A$. Then the plenary train roots of $A$ constitute exactly the Peirce spectrum of $P$ in \eqref{plenary1}. Applying \eqref{plenarypowers2} we obtain the Peirce symbol
\begin{equation}\label{Yplenary}
Y(\lambda,\mu,\nu)
=\sum_{k=1}^{n}
\gamma_{n-k}\frac{(2\nu)^k- (4\lambda\mu)^k}{\nu-2\lambda\mu}=\frac{ \varrho(P,\nu)-\varrho(P,2\lambda\mu)}{{\nu-2\lambda\mu}}
\end{equation}
 Similarly, using \eqref{Yplenary} one can derive the corresponding fusion laws in spirit of \cite{Fernand00}.

 \begin{remark}
 We point out that the explicit forms of the Peirce symbols $Y$ for principal and plenary powers given respectively by \eqref{Yplpl} and \eqref{Yplenary} are very different, but still have a common pattern in that their expressions can be written as a linear combination of terms
 $$
 \frac{ \varrho(P,\nu)-\varrho(P,\delta)}{{\nu-\delta}}
 $$
 for certain $\delta\in \sigma(P,c)$. It would be interesting to know whether the latter holds true in general.
 \end{remark}

\section{Hsiang algebras}\label{sec:Hsiang}

This class of algebras naturally emerges in the global geometry of minimal cones \cite{SimonV}, \cite{NTVbook}. The understanding of geometric and algebraic structure of minimal varieties is a challenging problem with various physical implications ranging from classical general relativity and brane physics \cite{Gibbons}. First  examples of algebraic degree three (cubic) minimal cones were constructed by Wu-Yi~ Hsiang \cite{Hsiang67} using the invariant theory. Hsiang also proposed a problem to classify all cubic minimal cones. Although the original problem lies in a pure geometrical context, the first progress in the classification  was achieved  by using representation theory of Clifford algebras  \cite{Tk10c}.  It can be shown \cite{NTVbook} that any cubic minimal cone carries a commutative nonassociative algebra structure  such that the defining polynomial of the minimal cone is the algebra norm. This bridges the analytic, geometric and algebraic faces of the problem. It turns out that classification of Hsiang algebras is intimately connected to Jordan and axial algebras structures, see \cite{Tk16}.

Below we apply our methods to derive the Peirce structure of Hsiang agebras.

\subsection{Metrized algebras}

An explicit definition of the underlying algebra structure is based on the concept of metrized algebras \cite{Bordemann}; we recall some standard facts on metrized algebras below.

A bilinear form $b(x,y)$ is called symmetric if $b(x,y)=b(y,x)$ and it is called  non-degenerate if its radical is trivial. A bilinear form $b(x,y)$ on an algebra $A$ is called \textit{associating} if
\begin{equation}\label{Qass}
b(xy,z)=b(x,yz), \quad \forall x,y,z\in A.
\end{equation}
An algebra carrying an associating non-degenerate symmetric bilinear form  is called \textit{metrized} \cite{Bordemann}.  The classical examples are the Killing form $\trace(\mathrm{ad}(x)\mathrm{ad}(y))$ of a Lie algebra and the invariant trace form $\trace L_{xy}$ of  a formal real (Euclidean) Jordan algebra \cite{FKbook}, \cite{Koecher}. Another important example is the Norton-Griess algebra $\mathfrak{G}$  appearing in connection with the Monster sporadic simple group \cite{Norton94} or, in general, many axial algebras \cite{HRS15}, \cite{Ivanov15}.

If one drops the non-degeneracy condition, the associating property \eqref{Qass} remains still interesting. Indeed, the associating forms in some axial algebras discussed in \cite{Rehren17}, \cite{HRS15} have nontrivial radical which depends on a concrete 3-transposition group representation.

An elementary observation which is also interesting in the context of this paper is the following  connection between algebras with associating symmetric bilinear form and baric algebras.

\begin{proposition}\label{pro:baric}
An algebra is baric if and only if it carries rank one associating symmetric bilinear form.
\end{proposition}

\begin{proof}
Let $A$ be a baric algebra and let  $\omega:A\to \Field$ be the baric homomorphism. Then the bilinear form $b_\omega(x,y)=\omega(x)\omega(y)$ is obviously  associating symmetric bilinear form of {rank one}.
In the converse direction, let $b(x,y)\not\equiv0$ be a rank one associating symmetric bilinear form  on an algebra $A$. Then $b(x,y)=v(x)v(y)$ for some linear form $v:A\to \Field$. Let $z\in A$ satisfy $v(z)\ne0$ (such a $z$ exists because $b\not\equiv0$). Using
$$
v(xy)v(z)=b(xy,z)=b(x,yz)=v(x)v(yz)=b(y,xz)=v(y)v(xz),
$$
we obtain
$$
v(xz)v(z)=v(x)v(z^2) \quad \Rightarrow\quad
v(xz)=\frac{v(z^2)}{v(z)}v(x)
$$
therefore
$$
v(xy)=\frac{v(z^2)}{v^2(z)}v(x)v(y).
$$
Since $v\not\equiv0$, it follows that $v(z^2)\ne0$. This implies that  $\omega(x)=\frac{v(z^2)}{v^2(z)}v(x)$ is an nontrivial algebra homomorphism:
$$
\omega(xy)=\frac{v(z^2)}{v^2(z)}v(xy)=\frac{v(z^4)}{v^4(z)}v(x)v(y)= \omega(x)\omega(y),
$$
hence the algebra $A$ is baric.
\end{proof}

\subsection{Hsiang algebras}

\begin{definition}[\cite{Tk16}, \cite{NTVbook}] A \textit{Hsiang algebra} is a metrized commutative nonassociative algebra over $\R{}$ satisfying the identities:
\begin{equation}\label{Hsiangtrace}
\trace L_z=0, \quad \forall z\in A,
\end{equation}
and
\begin{equation}\label{Hsiang1}
a_1zz^3+a_2z^2z^2-a_3 b(z,z)z^2-a_4b(z,z^2)z=0,
\end{equation}
where $b(x,y)$ is a positive definite associating symmetric bilinear form, and $a_i\in\R{}$ subject to the nondegenracy condition $$
(a_1+a_2)(a_3+a_4)\ne0.
$$
\end{definition}

 The above algebra defining identity \eqref{Hsiang1} is a direct translation of the minimal surface equation (a certain nonlinear  partial differential equation) such that the defining equation of a minimal cone is written as a generic norm in a certain commutative nonassociative algebra. Explicit examples of Hsiang algebras will be given in section~\ref{sec:exHs} below. In short,  the connection between the geometric and algebraic sides of Hsiang algebras is very simple: one can show that given a Hsiang algebra $A$, the zero locus $f_A^{-1}(0)$ of the cubic form
$$
f_A(z):=b(z^2,z)
$$
is a minimal cone in the Euclidean vector space $A$ with the metric $b(z,z)$. Conversely, any minimal cone on a Euclidean vector space $V$ with an inner product $b$ satisfying a cubic polynomial   identity $f(z)=0$  gives rise in essentially unique manner to a  metrized commutative algebra on $V$ such that $f=f_A$ and \eqref{Hsiangtrace}-\eqref{Hsiang1} hold. We omit the derivation and  refer the interested readers to Chapter~6 in a recent monograph  \cite{NTVbook} for further details.

It can be shown that the zero locus (the minimal cone corresponding to $A$)
$$
A_0:=\{z\in A:b(z^2,z)=0\}
$$
is always a nontrivial real algebraic variety containing nonzero points.
Two Hsiang algebras are called equivalent if their zero loci are congruent, i.e. coincide upon a $b$-isometry. For the homogeneity reasons, any dilatation  $b\to kb$, where $0\ne k\in \R{}$, preserves the zero locus, hence the corresponding Hsiang algebras are equivalent. Taking the inner product with $x$ in \eqref{Hsiang1} and applying \eqref{Qass}, we find
%$$
%a_1b(zz^3,z)+a_2b(z^2z^2,z)+a_3 b(z,z)b(z,z^2)+a_4b(z,z^2)b(z,z)=0,
%$$
%which yields by virtue of \eqref{Qass} that
$$
(a_1+a_2)b(z^2z^2,z)=(a_3+a_4)b(z,z^2)b(z,z),
$$
therefore by virtue of $(a_1+a_2)(a_3+a_4)\ne0$ we have
\begin{equation}\label{Hsiangk}
b(z^2z^2,z)=kb(z,z^2)b(z,z)
\end{equation}
for some $0\ne k\in \R{}$. Replacing $b$ by $kb$, we may assume by the made above remarks that $k=1$. Finally, linearizing
\begin{equation}\label{Hsiang2}
b(z^2z^2,z)=b(z,z^2)b(z,z),
\end{equation}
we arrive at
\begin{equation}\label{Hsiang3}
4z^4+z^2z^2-3 b(z,z)z^2-2b(z,z^2)z=0.
\end{equation}
This shows that any Hsiang algebra is equivalent to  an algebra satisfying \eqref{Hsiang3} and \eqref{Hsiangtrace}. Note also that any metrized algebra satisfying \eqref{Hsiangk} is Hsiang (the proof is by linearization of \eqref{Hsiangk}).

We call any algebra satisfying \eqref{Hsiang3} a \textit{normalized} Hsiang algebra.

%Equations \eqref{Hsiang2} and \eqref{Hsiang3} are equivalent if one assumes that $A$ satisfies \eqref{Qass}. We call any algebra satisfying \eqref{Hsiang3} a normalized Hsiang algebra.

Suppose for simplicity that $A$ is normalized. Since any Hsiang algebra is metrized, it contains  nonzero idempotents, see \cite{Tk15b}, \cite{Tk18a}, \cite{Tk18b}. Let $c\ne0$ be an idempotent in $A$. Then it follows from \eqref{Hsiang2} that $b(c,c)=1$. In particular, the Peirce polynomial is independent on a choice of an idempotent. Identifying
$$
\phi_{zz^3}=4, \quad
\phi_{z^2z^2}=1,\quad
\phi_{z^2}=-3 b(z,z),\quad
\phi_{z}=-2b(z,z^2),
$$
we obtain
$$
\phi_{zz^3}(c)=4, \quad
\phi_{z^2z^2}(c)=1,\quad
\phi_{z^2}(c)=-3,\quad
\phi_{x}(c)=-2,
$$
therefore  the Peirce polynomial is found from \eqref{Hsiang3} to be
\begin{equation}\label{Hsiangpeirce}
\begin{split}
\varrho(P,q)&=4\varrho(xx^3,q)+\varrho(x^2x^2,q)-3\varrho(x^2,q) -2\varrho(x,q)\\
&=2(4q^3+4q^2-q-1)\\
&=2(2q-1)(2q+1)(q+1).
\end{split}
\end{equation}
The Peirce spectrum is given by
$$
\sigma(P,c)=\{-1,-\half12,\half12\}
$$
and \textit{does not depend on a choice of an idempotent}.
Note also that $\lambda=1$ does not belong $\sigma(P,c)$, hence any idempotent in a Hsiang algebra is always primitive. Since $b$ is positive definite (in particular non-degenerate) and $\Field=\R{}$, it follows that $c$ is  semi-simple. Hence we arrive at the Peirce decomposition
$$
A=A_c(1) \oplus A_c(-1)\oplus A_c(-\half12)\oplus A_c(\half12), \qquad \text{where }A_c(1)=\R{}c.
$$
It also follows from \eqref{Qass} that the latter direct decomposition is orthogonal with respect to $b$. Furthermore, the traceless condition ~\eqref{Hsiangtrace} implies the following obstructions on the Peirce dimensions:
\begin{equation}\label{I90}
n_3(c)=2n_1(c)+n_2(c)-2
\end{equation}
and
\begin{equation}\label{I80}
\dim A=3n_1(c)+2n_2(c)-1,
\end{equation}
where
$$
\phantom{\frac{1}{3}}n_1(c)=\dim A_c(-1),\quad n_2(c)=\dim A_c(-\half{1}{2}),\quad n_3(c)=\dim A_c(\half{1}{2}).
$$

Next, in order to derive the fusion laws, we consider the Peirce symbol
\begin{equation}\label{Yeq}
Y(\lambda,\mu,\nu)=8(\nu^2+\lambda^2+\mu^2)+ 8(\nu\lambda+\nu\mu+ \lambda\mu)+4(\lambda+\mu+\nu)-6.
\end{equation}
Note that $Y(\lambda,\mu,\nu)$ is completely symmetric in the three arguments (as a corollary of \eqref{Qass}). We also have $\delta_1\phi_{z^2z^2}=\delta_1\phi_{z^3z}=0$ and
$$
\delta_1\phi_{z^2}(x;u)=-6b(x,u), \qquad
\delta_1\phi_{z}(x;u)=-6b(x^2,u),
$$
hence
\begin{align*}
\tau_1(u,a)&=\delta_1\phi_{z^2}(c;u)\,\varrho(z^2,a)+ \delta_1\phi_{z}(c;u)\,\varrho(z,a)=-6b(c,u)(1+2a).
\end{align*}
Since $c$ and $A_c(\lambda)$ are orthogonal, we have $\tau_1(x,\lambda)=0$ for any $x\in A_c(\lambda)$, where $\lambda\in \{-1,-\frac12,\frac12\}$.
Furthermore, $\delta_2\phi_{z^2}(x;u,v)=-6$  and $\delta_2\phi_{z}(x;u,v)=-12b(xu,v)$, hence
$\tau_2(u,v)= -12b(cu,v).$
In summary, applying \eqref{Qeq}  we obtain for any $u\in A_c(\lambda)$, $v\in A_c(\mu)$
\begin{equation}\label{Y1}
\sum_{\nu\in \{1,-1,-\frac12,\frac12\}} Y(\lambda,\mu,\nu)(uv)_{\nu}=
(12\lambda+6)b(u,v)c.
\end{equation}
\begin{table}[ht]
	%\begin{center}
	\renewcommand{\arraystretch}{1.4}
	\setlength{\tabcolsep}{0.85em}
	\begin{tabular}{r|rrrr}
			$\star$	& $1$ & $-1$ & $-\half12$ & $\half12$\\
		\hline
			$1$ & $1$ & $-1$ & $-\half12$ & $\half12$\\
			$-1$ &	& $1$ & $\half12$ &$-\half12,\half12$\\
			$-\half12$ &	&	& $1,-\half12$ & $-1,\half12$\\
            $\half12$ &	&	& & $1,-1,-\half12$
	\end{tabular}
	%\end{center}
\smallskip
	\caption{Hsiang algebra fusion rules}
	\label{tab2}
    \end{table}

For example, setting $u,v\in A_c(-1)$ and using \eqref{Yeq} yields
$$
6(uv)_{1}+30(uv)_{-1}+18(uv)_{-\frac12}+6(uv)_{\frac12}=-6b(u,v)c,
$$
which immediately implies that $(uv)_{-1}=(uv)_{-\frac12}=(uv)_{\frac12}=0$, i.e. $A_c(-1)A_c(-1)\subset \R{}c$, therefore
$$
uv=(uv)_{1}=-b(u,v)c.
$$
Similarly, if $u,v\in A_c(-\frac12)$  then
$$
8(uv)_{-1}-4(uv)_{-\frac12}=0,
$$
which implies $(uv)_{-1}=(uv)_{-\frac12}=0$, thus,
$$
A_c(-\half12)A_c(-\half12)\subset \R{}c\oplus A_c(\half12).
$$
A similar argument easily yields the fusion laws of $A$ shown in Table~\ref{tab2}.

\subsection{Hsiang algebras from Jordan algebras}\label{sec:exHs}

Some examples of Hsiang algebras can be obtained by contraction of Jordan algebras on their subspaces. Let us recall the definition of contraction given by Griess \cite{Griess2010}. Let  $A$ be any algebra, $\pi:A\to B$ be a projection of vector spaces, i.e. $\pi^2=\pi$ and $\pi(A)=B$. Define an algebra structure on $B$ by
$$
x\circ y=\pi(x\cdot y), \qquad x,y\in B,
$$
where $x\cdot x=xy$ means the product in $A$. The algebra $A_\pi:=(B,\circ)$ is called the \textit{contraction} of $A$
to $B$ with respect to the projection $\pi$. Note that if $A$ is commutative, so is the contraction. Furthermore, if $A$ is metrized then so is $B$. Indeed, for any $x,y,z\in B$ we have $x\cdot y=x\circ y+h$ for some $h\in H$, hence
\begin{equation}\label{BH}
b(x\cdot y,z)=b(x\circ y+h,z)=b(x\circ y,z),
\end{equation}
therefore $b(x\cdot y,z)=b(x,y\cdot z)$ implies $b(x\circ y,z)=b(x,y\circ z)$.

Now suppose that $H$ is a nonzero subalgebra of a metrized  algebra $A$ and let
\begin{equation}\label{ABH}
A=B\oplus H
\end{equation}
be the $b$-orthogonal decomposition as vector spaces. Let us consider the contraction of $A$ on $B$. Since $HH\subset H$ we  have $b(x\cdot h_1,h_2)=b(x,h_1\cdot h_2)=0$
for any $x\in B$ and $h_i\in H$. This implies
$$
B\cdot H\subset B.
$$
Using \eqref{BH} we have for any $x,y\in B$
\begin{equation}\label{by}
b(y\circ y,x)=b(y\cdot y,x)=
%b(y,x\circ y)=b(x,y^2)=
b(y,y\cdot x).
\end{equation}

Now suppose  that $A$ is a finite-dimensional formal real Jordan algebra over $\R{}$ with unit $e$.  Any Jordan algebra is power-associative and satisfies a polynomial relation (given by the so-called generic minimal polynomial \cite[p.~452]{Knus}). The  degree of the generic minimal polynomial is called the {rank} of a Jordan algebra. In what follows we assume  that $A$ \textit{has rank three}, i.e. any element $x\in A$ is annihilated by a cubic polynomial:
\begin{equation}\label{minimalpol}
x^3-\alpha(x)x^2+\beta(x)x-\gamma(x)e=0.
\end{equation}
If $A$ is formally real then the generic trace form $\alpha(x)$ gives rise to an associating positive definite bilinear form
$$
b(x,y):=\alpha(xy)
$$
on $A$ \cite{Koecher}, \cite{FKbook}, thus  $A$ is a metrized algebra.

\begin{theorem}
Let $A$ be a rank three formally real simple Euclidean Jordan algebra and $H\subset A$ be a subalgebra of $A$. Then the contraction of $A$ on $B=H^\bot$ is a Hsiang algebra.
\end{theorem}

\begin{proof}
It is well-known  and easily follows from \eqref{minimalpol} that $\beta(x)$ and $\gamma(x)$ are explicitly determined by $\alpha$ by virtue of the Newton identities
\begin{equation}\label{Newton}
\begin{split}
\beta(x)&=\half12(\alpha(x)^2-\alpha(x^2)),\\
\gamma(x)&=\half16(\alpha^3(x)-3\alpha(x^2)\alpha(x)+2\alpha(x^3)).
\end{split}
\end{equation}
Let $H\subset A$ be a nonzero subalgebra and let $B=H^\bot$ (such that \eqref{ABH} holds) be the contraction of $A$ with respect to the $b$-orthogonal projection. By the above, $b$ is also positive definite and associating on $B$. Since $e\in H$ we have
\begin{equation}\label{be0}
0=b(x,e)=\alpha(xe)=\alpha(x),\qquad \text{ for all $x\in B$.}
\end{equation}
Therefore $\beta(x)=-\frac12\alpha(x^2)=-\frac12b(x,x)$ and $\gamma(x)=\frac13\alpha(x^3)=\frac13b(x,x^2)$ on $B$, hence
\begin{equation}\label{minimalpolA}
x^3=\half13b(x,x^2)e+\half12b(x,x)x
\end{equation}

Now, specializing $y=x^{\circ 2}=x\circ x=x^2-h$ in \eqref{by}, where $h\in H$, we obtain
\begin{equation}\label{WBOT2}
\begin{split}
b(x^{\circ 2}\circ x^{\circ 2},x)&=
b(x^{\circ 2},x\cdot x^{\circ 2})=
b(x^{\circ 2},x^3-x\cdot h)=
b(x^{\circ 2},x^3)-b(x^{\circ 2},x\cdot h)=
\\
&
=b(x^2-h,x^3)-b(x^2-h,x\cdot h)=
b(x^2,x^3)-2b(x^3,h)
\end{split}
\end{equation}
Using \eqref{minimalpolA} and \eqref{be0},
\begin{align*}
b(x^3,h)&=\half13b(x,x^2)b(e,h)+\half12b(x,x)b(x,h)
=\half13b(x,x^2)b(e,x^2-x^{\circ2})\\
&=\half13b(x,x^2)b(e,x^2)=
\half13b(x,x^2)b(x,x),
\end{align*}
and similarly
\begin{align*}
b(x^2,x^3)=\half13b(x,x^2)b(x^2,e)+\half12b(x,x)b(x^2,x)= \half56b(x,x)b(x^2,x).
\end{align*}
Combining the obtained expressions with \eqref{WBOT2} we obtain
$$
b(x^{\circ 2}\circ x^{\circ 2},x)=\half16b(x,x)b(x^2,x)=
\half16b(x,x)b(x\circ x,x),
$$
which proves \eqref{Hsiangk} with $k=\frac16$.

Finally, suppose that $\{e_i\}_{1\le i\le k}$ is an orthonormal basis in $B$, and let $\{e_i\}_{k+1\le i\le n}$ be the completion to an orthonormal basis of $A$ with $e_n=e$. Since $A$ is simple,
$$
\sum_{i=1}^ne_i^2=e
$$
see for example 6b) on p.~59 in \cite{FKbook}, therefore
$$
\sum_{i=1}^ke_i^2=e-\sum_{i=k+1}^ne_i^2\in H,
$$
hence if $L_{\circ x}:y\to x\circ y$ the multiplication operator on $B$, we obtain for any $x\in B$
$$
\trace L_{\circ x}=\sum_{i=1}^kb(L_{\circ x}e_i, \, e_i)=
b\left(\sum_{i=1}^k e_i\circ e_i, x\right)=0.
$$
the latter proves \eqref{Hsiangtrace}. The theorem is proved.
\end{proof}

\begin{remark}
It follows from the proof of the theorem that the simplicity of $A$ is only needed for the trace-free property, while the defining relation \eqref{Hsiangk} is valid for a general Jordan algebra $A$.
\end{remark}

\begin{example}
Let $A=\mathfrak{h}_r(\mathbb{A}_d)$ be the algebra over the reals on the vector space of all Hermitian matrices of size $r$ over the Hurwitz division algebra $\F_d$, $d\in \{1,2,4,8\}$ (the reals $\F_1=\R{}$, the complexes $\F_2=\mathbb{C}$, the Hamilton quaternions $\F_4=\mathbb{H}$ and the Graves-Cayley octonions $\F_8=\mathbb{O}$) with the multiplication
$$
x\cdot y=\half12(xy+yx),
$$
where $xy$ is the standard matrix multiplication. Then it is classically known that if $r\le 3$ and $d\in \{1,2,4,8\}$, or if $r\ge 4$ and $d\in \{1,2,4\}$ then $A$ is a \textit{simple} Jordan algebra with the unit matrix $e$ being the algebra unit. In that case $\alpha(x)=\trace x$ is the trace of the matrix $x$.

The simplest example is the $\frac{(r-1)(r+2)}{2}$-dimensional commutative algebra $B$ obtained by the contraction of $A=\mathfrak{h}_r(\mathbb{A}_1)$ with $H=\Span(e)$. In other words,  $B$ is obtained as the vector space of all real trace free matrices of size $r\ge 2$ with the multiplication
$$
x\circ y=\frac12(xy+yx)-\frac{\trace xy}{r}e.
$$
\end{example}

\subsection{Nourigat--Varro algebras}
%\begin{remark}\label{rem:Nourigat}
Finally, we want to mention that very recently some similar algebraic structures also appeared independently in the context of $\omega$-PI algebras of degree four in the works of by M. Nourigat  and R. Varro  \cite{NourigatI}, \cite{NourigatII}, \cite{NourigatIII}. A model example there is commutative \textit{baric} algebras satisfying identity
$$
a_1 z^2z^2 + a_2z^4 = b_1\omega(z)z^3+ b_2\omega^2(z)z^2+b_3\omega^3(z)z.
$$
The latter identity can be thought as an analogue of the Hsiang identity for a degenerate associating form $b$. Indeed, by Proposition~\ref{pro:baric} $g(x,y)=\omega(x)\omega(y)$ is an associating bilinear form of rank 1. Therefore it is natural to combine these cases, to a general equation
\begin{equation}\label{general}
a_1z^4+a_2z^2z^2=b_1g(z,z)z^2+b_2g(z,z^2)z,
\end{equation}
where $g(x,y)$ is an \textit{associating} symmetric form,  not necessarily nondegenerate. If $g$ is positive definite then one arrives at the defining identity  \eqref{Hsiang3}, while $g(x,y)=\omega(x)\omega(y)$  fits the definition of the Nourigat--Varro algebras. Some other relevant examples of identities of degree four can be found in \cite{ElOkubo}, \cite{EldLabra}.

\bibliographystyle{plain}%{alpha}%{amsalpha}%

\begin{thebibliography}{10}

\bibitem{Albert48}
A.~A. Albert.
\newblock Power-associative rings.
\newblock {\em Trans. Amer. Math. Soc.}, 64:552--593, 1948.

\bibitem{Audibert}
Pierre Audibert.
\newblock {\em Mathematics for informatics and computer science}.
\newblock John Wiley \& Sons, 2013.

\bibitem{Bayara10}
Joseph Bayara, Andr\'e Conseibo, Moussa Ouattara, and Fouad Zitan.
\newblock Power-associative algebras that are train algebras.
\newblock {\em J. Algebra}, 324(6):1159--1176, 2010.

\bibitem{Bordemann}
M.~Bordemann.
\newblock Nondegenerate invariant bilinear forms on nonassociative algebras.
\newblock {\em Acta Math. Univ. Comenian.}, 66(2):151--201, 1997.

\bibitem{DeMedts17}
Tom De~Medts and Felix Rehren.
\newblock Jordan algebras and 3-transposition groups.
\newblock {\em J. Algebra}, 478:318--340, 2017.

\bibitem{ElOkubo}
A.~Elduque and S.~Okubo.
\newblock On algebras satisfying {$x^2x^2=N(x)x$}.
\newblock {\em Math. Z.}, 235(2):275--314, 2000.

\bibitem{EldLabra}
Alberto Elduque and Alicia Labra.
\newblock On some {J}ordan baric algebras.
\newblock {\em J. Algebra Appl.}, 12(5):1250215, 12, 2013.

\bibitem{Etherington}
I.~M.~H. Etherington.
\newblock Genetic algebras.
\newblock {\em Proc. Roy. Soc. Edinburgh}, 59:242--258, 1939.

\bibitem{Etherington49}
I.~M.~H. Etherington.
\newblock Non-associative arithmetics.
\newblock {\em Proc. Roy. Soc. Edinburgh. Sect. A.}, 62:442--453, 1949.

\bibitem{Etherington60}
I.~M.~H. Etherington.
\newblock Enumeration of indices of given altitude and degree.
\newblock {\em Proc. Edinburgh Math. Soc. (2)}, 12:1--5, 1960/1961.

\bibitem{Etherington62}
I.~M.~H. Etherington.
\newblock Note on quasigroups and trees.
\newblock {\em Proc. Edinburgh Math. Soc. (2)}, 13:219--222, 1962/1963.

\bibitem{FKbook}
J.~Faraut and A.~Kor{\'a}nyi.
\newblock {\em Analysis on symmetric cones}.
\newblock Oxford Math. Monographs. 1994.
\newblock Oxford Sci. Publ.

\bibitem{Gibbons}
G.W. Gibbons, K.-I. Maeda, and U.~Miyamoto.
\newblock The {B}ernstein conjecture, minimal cones and critical dimensions.
\newblock {\em Classical Quantum Gravity}, 26(18):185008, 14, 2009.

\bibitem{Griess82}
Robert~L. Griess, Jr.
\newblock The friendly giant.
\newblock {\em Invent. Math.}, 69(1):1--102, 1982.

\bibitem{Griess85}
Robert~L. Griess, Jr.
\newblock The {M}onster and its nonassociative algebra.
\newblock In {\em Finite groups---coming of age ({M}ontreal, {Q}ue., 1982)},
  volume~45 of {\em Contemp. Math.}, pages 121--157. Amer. Math. Soc.,
  Providence, RI, 1985.

\bibitem{Griess2010}
Robert~L. Griess, Jr.
\newblock Nonassociativity in {VOA} theory and finite group theory.
\newblock {\em Comment. Math. Univ. Carolin.}, 51(2):237--244, 2010.

\bibitem{Fernand00}
J.~Carlos Guti\'errez~Fern\'andez.
\newblock Principal and plenary train algebras.
\newblock {\em Comm. Algebra}, 28(2):653--667, 2000.

\bibitem{Guzzo94}
Henrique Guzzo, Jr.
\newblock The {P}eirce decomposition for commutative train algebras.
\newblock {\em Comm. Algebra}, 22(14):5745--5757, 1994.

\bibitem{HRS15b}
J.I. Hall, F.~Rehren, and S.~Shpectorov.
\newblock Primitive axial algebras of {J}ordan type.
\newblock {\em J. Algebra}, 437:79--115, 2015.

\bibitem{HRS15}
J.I. Hall, F.~Rehren, and S.~Shpectorov.
\newblock Universal axial algebras and a theorem of {S}akuma.
\newblock {\em J. Algebra}, 421:394--424, 2015.

\bibitem{HSS18}
J.I. Hall, Y.~Segev, and S.~Shpectorov.
\newblock Miyamoto involutions in axial algebras of {J}ordan type half.
\newblock {\em Israel J. Math.}, 223(1):261--308, 2018.

\bibitem{HSS18b}
J.I. Hall, Y.~Segev, and S.~Shpectorov.
\newblock On primitive axial algebras of {J}ordan type.
\newblock {\em Bull. Inst. Math. Acad. Sin. (N.S.)}, 13, 2018.
\newblock to appear.

\bibitem{Holgate}
P.~Holgate.
\newblock Genetic algebras satisfying {B}ernstein's stationarity principle.
\newblock {\em J. London Math. Soc. (2)}, 9:612--623, 1974/75.

\bibitem{Hsiang67}
W.-Y. Hsiang.
\newblock Remarks on closed minimal submanifolds in the standard {R}iemannian
  {$m$}-sphere.
\newblock {\em J. Differential Geometry}, 1:257--267, 1967.

\bibitem{Ivanov09}
A.~A. Ivanov.
\newblock {\em The {M}onster group and {M}ajorana involutions}, volume 176 of
  {\em Cambridge Tracts in Math.}
\newblock Cambridge Univ. Press, Cambridge, 2009.

\bibitem{Ivanov15}
A.~A. Ivanov.
\newblock Majorana representation of the {M}onster group.
\newblock In {\em Finite simple groups: thirty years of the atlas and beyond},
  volume 694 of {\em Contemp. Math.}, pages 11--17. Amer. Math. Soc.,
  Providence, RI, 2017.

\bibitem{Knus}
M.-A. Knus, A.~Merkurjev, M.~Rost, and J.-P. Tignol.
\newblock {\em The book of involutions}, volume~44 of {\em American
  Mathematical Society Colloquium Publications}.
\newblock American Mathematical Society, Providence, RI, 1998.
\newblock With a preface in French by J. Tits.

\bibitem{Koecher}
M.~Koecher.
\newblock {\em The {M}innesota notes on {J}ordan algebras and their
  applications}, volume 1710 of {\em Lecture Notes in Mathematics}.
\newblock Springer-Verlag, Berlin, 1999.
\newblock Edited, annotated and with a preface by Aloys Krieg and Sebastian
  Walcher.

\bibitem{KrTk18a}
Ya. Krasnov and V.G. Tkachev.
\newblock Variety of idempotents in nonassociative algebras.
\newblock In {\em Modern Trends in Hypercomplex Analysis}, Trends Math.
  Birkh\"auser/Springer, 2018.
\newblock to appear. Available on arXiv:1801.00617.

\bibitem{Kurosh47}
A.~Kurosh.
\newblock Non-associative free algebras and free products of algebras.
\newblock {\em Rec. Math. [Mat. Sbornik] N.S.}, 20(62):239--262, 1947.

\bibitem{Lyubich92}
Yuri~I. Lyubich.
\newblock {\em Mathematical structures in population genetics}, volume~22 of
  {\em Biomathematics}.
\newblock Springer-Verlag, Berlin, 1992.
\newblock Translated from the 1983 Russian original by D. Vulis and A. Karpov.

\bibitem{Mallol15}
Cristi\'an Mallol and Richard Varro.
\newblock Homogam\'etisation d'alg\`ebres pond\'er\'ees.
\newblock {\em J. Algebra}, 427:1--19, 2015.

\bibitem{Mallol17}
Cristi\'an Mallol and Richard Varro.
\newblock Sur les identit\'es polynomiales v\'erifi\'ees par les alg\`ebres de
  r\'etrocroisement.
\newblock {\em Comm. Algebra}, 45(8):3555--3586, 2017.

\bibitem{McCr65}
K.~McCrimmon.
\newblock Norms and noncommutative {J}ordan algebras.
\newblock {\em Pacific J. Math.}, 15:925--956, 1965.

\bibitem{McCrbook}
K.~McCrimmon.
\newblock {\em A taste of {J}ordan algebras}.
\newblock Universitext. Springer-Verlag, New York, 2004.

\bibitem{Meyberg1}
K.~Meyberg and J.~M. Osborn.
\newblock Pseudo-composition algebras.
\newblock {\em Math. Z.}, 214(1):67--77, 1993.

\bibitem{NTVbook}
N.~Nadirashvili, V.G. Tkachev, and S.~Vl{\u{a}}du{\c{t}}.
\newblock {\em Nonlinear elliptic equations and nonassociative algebras},
  volume 200 of {\em Math. Surveys and Monographs}.
\newblock AMS, Providence, RI, 2014.

\bibitem{Norton94}
S.~Norton.
\newblock The {M}onster algebra: some new formulae.
\newblock In {\em Moonshine, the {M}onster, and related topics ({S}outh
  {H}adley, {MA}, 1994)}, volume 193 of {\em Contemp. Math.}, pages 297--306.
  Amer. Math. Soc., Providence, RI, 1996.

\bibitem{NourigatI}
Michelle Nourigat and Richard Varro.
\newblock \'etude des {$\omega$}-{PI} alg\`ebres commutatives de degr\'e 4:
  {I}. {A}lg\`ebres non barycentriques non invariantes par gam\'etisation.
\newblock {\em Comm. Algebra}, 39(11):3956--3968, 2011.

\bibitem{NourigatII}
Michelle Nourigat and Richard Varro.
\newblock \'etudes des {$\omega$}-{PI} alg\`ebres commutatives de degr\'e 4:
  {II}. {A}lg\`ebres non barycentriques invariantes par gam\'etisation.
\newblock {\em Comm. Algebra}, 39(8):2764--2778, 2011.

\bibitem{NourigatIII}
Michelle Nourigat and Richard Varro.
\newblock \'etude des {$\omega$}-{PI} alg\`ebres commutatives de degr\'e 4:
  {III}. {A}lg\`ebres barycentriques invariantes par gam\'etisation.
\newblock {\em Comm. Algebra}, 41(8):2825--2851, 2013.

\bibitem{Osborn}
J.~Marshall Osborn.
\newblock Varieties of algebras.
\newblock {\em Advances in Math.}, 8:163--369 (1972), 1972.

\bibitem{Reed}
M.~L. Reed.
\newblock Algebraic structure of genetic inheritance.
\newblock {\em Bull. Amer. Math. Soc.}, 34(2):107--130, 1997.

\bibitem{Rehren17}
F.~Rehren.
\newblock Generalised dihedral subalgebras from the {M}onster.
\newblock {\em Trans. Amer.Math.Soc.}, 369(10):6953--6986, 2017.

\bibitem{Rowenbook}
Louis~Halle Rowen.
\newblock {\em Graduate algebra: noncommutative view}, volume~91 of {\em
  Graduate Studies in Mathematics}.
\newblock American Mathematical Society, Providence, RI, 2008.

\bibitem{Schafer}
R.D. Schafer.
\newblock {\em An introduction to nonassociative algebras}.
\newblock Pure and Applied Mathematics, Vol. 22. Academic Press, New York,
  1966.

\bibitem{SimonV}
Leon Simon.
\newblock The minimal surface equation.
\newblock In {\em Geometry, {V}}, volume~90 of {\em Encyclopaedia Math. Sci.},
  pages 239--272. Springer, Berlin, 1997.

\bibitem{Tk10c}
V.G. Tkachev.
\newblock Minimal cubic cones via {C}lifford algebras.
\newblock {\em Complex Anal. Oper. Theory}, 4(3):685--700, 2010.

\bibitem{Tk16}
V.G. Tkachev.
\newblock Hsiang algebras and cubic minimal cones.
\newblock {\em In progress, unpublished manuscript, 140 p.}, 2016.

\bibitem{Tk15b}
V.G. Tkachev.
\newblock On the non-vanishing property for real analytic solutions of the
  $p$-{L}aplace equation.
\newblock {\em Proc. Amer. Math. Soc.}, 144(6):2375--2382, 2016.

\bibitem{Tk18a}
V.G. Tkachev.
\newblock A correction of the decomposability result in a paper by
  {M}eyer-{N}eutsch.
\newblock {\em J. Algebra}, 504:432--439, 2018.

\bibitem{Tk18b}
V.G. Tkachev.
\newblock On an extremal property of {J}ordan algebras of {C}lifford type.
\newblock {\em Comm. Alg.}, 2018.
\newblock to appear, available on arXiv:1801.05724.

\bibitem{Walcher1}
S.~Walcher.
\newblock On algebras of rank three.
\newblock {\em Comm. Algebra}, 27(7):3401--3438, 1999.

\bibitem{WBusekros}
Angelika W\"orz-Busekros.
\newblock {\em Algebras in genetics}, volume~36 of {\em Lecture Notes in
  Biomathematics}.
\newblock Springer-Verlag, Berlin-New York, 1980.

\bibitem{Zhevlakov}
K.~A. Zhevlakov, A.~M. Slin\cprime~ko, I.~P. Shestakov, and A.~I. Shirshov.
\newblock {\em Rings that are nearly associative}, volume 104 of {\em Pure and
  Applied Mathematics}.
\newblock Academic Press, Inc., New York-London, 1982.
\newblock Translated from the Russian by Harry F. Smith.

\end{thebibliography}
\def\cprime{$'$}

\end{document}